\documentclass[12pt]{amsart}
\usepackage{amssymb,accents, caption, subcaption}
\usepackage{amsfonts}
\usepackage{latexsym}
\usepackage{amscd}
\usepackage[mathscr]{euscript}
\usepackage{xy} \xyoption{all}
\usepackage[active]{srcltx}
\usepackage{color,graphics}
\usepackage{tikz,tkz-graph,tikz-cd}
\usetikzlibrary{matrix,arrows,calc,automata,trees}
\usepackage[shortlabels]{enumitem}
\usetikzlibrary{decorations.pathreplacing}

\newcommand{\exend}{\unskip\nobreak\hfill$\blacktriangleleft$}
\DeclareMathSymbol{\widehatsym}{\mathord}{largesymbols}{"62}

\newcommand{\act}{\curvearrowright}
\newcommand{\andspace}{\quad \text{and} \quad}
\newcommand{\If}{\textup{if }}
\newcommand{\ini}{\mathfrak{i}}
\newcommand{\ter}{\mathfrak{t}}
\newcommand{\niceblue}{rgb:red,1;green,2;blue,3}
\newcommand{\nicegreen}{rgb:red,1;green,4;blue,2}
\newcommand{\nicered}{rgb:red,4;green,1;blue,2}

\DeclareMathOperator{\id}{id} 

\theoremstyle{plain}
\newtheorem{theorem}{Theorem}[section]
\newtheorem{lemma}[theorem]{Lemma}
\newtheorem{proposition}[theorem]{Proposition}
\newtheorem{corollary}[theorem]{Corollary}

\theoremstyle{definition}
\newtheorem{definition}[theorem]{Definition}
\newtheorem{example}[theorem]{Example}

\newtheorem{remark}[theorem]{Remark}
\newtheorem*{remark*}{Remark}

\newtheorem*{assumption*}{Assumption}

\newtheorem{notation}[theorem]{Notation}

\begin{document}

\null\vskip-1cm

\title[Exchange rings and real rank zero $C^*$-algebras]{Exchange rings and real rank zero $C^*$-algebras associated with finitely separated graphs}%
\author{Matias Lolk}
\address{Department of Mathematical Sciences, University of Copenhagen, 2100 Copenhagen, Denmark.}\email{lolk@math.ku.dk}
\thanks{Supported by the Danish National Research Foundation through the Centre for Symmetry and Deformation (DNRF92)} 
\date{\today}

\begin{abstract}
We introduce a generalisation of Condition (K) to finitely separated graphs and show that it is equivalent to essential freeness of the associated partial action as well as the exchange property of any of the associated tame algebras. As a consequence, we can show that any tame separated graph algebra with the exchange property is separative.
\end{abstract}

\maketitle


\section*{Introduction}
A \textit{finitely separated graph} is a directed graph with a partition of the edges into finite subsets, which might be thought of as an edge-colouring, so that edges with distinct ranges have different colours. Ara and Goodearl first introduced \textit{Leavitt path algebras} $L_K(E,C)$ and \textit{graph $C^*$-algebras} $C^*(E,C)$ associated with a separated graph $(E,C)$ in \cite{AG2} and \cite{AG}, respectively, and showed that any conical abelian monoid may be realised as the non-stable $K$-theory $\mathcal{V}(L_K(E,C))$ of a Leavitt path algebra of a finitely separated graph. They also conjectured that the inclusion $L_\mathbb{C}(E,C) \hookrightarrow C^*(E,C)$ induces an isomorphism on non-stable $K$-theory, but this important problem remains open. While the edges and their adjoints define partial isometries in these algebras, their products need not be partial isometries -- we say that $E^1$ is not a \textit{tame} set of partial isometries in $L_K(E,C)$ and $C^*(E,C)$. This led Ara and Exel to define quotients $L_K(E,C) \twoheadrightarrow L_K^{\textup{ab}}(E,C)$ and $C^*(E,C) \twoheadrightarrow \mathcal{O}(E,C)$ in which $E^1$ is exactly forced to be tame, as well as a further \textit{reduced quotient} $\mathcal{O}^r(E,C)$ in the $C^*$-setting. Amazingly, passing to these much more well behaved quotients only enriches the non-stable $K$-theory (at least on the level of Leavitt path algebras) in the sense that the induced monoid homomorphism $\mathcal{V}(L_K(E,C)) \to \mathcal{V}(L_K^{\textup{ab}}(E,C))$ is a refinement. Moreover, if the above conjecture holds, then the canonical embedding $L_\mathbb{C}^{\textup{ab}}(E,C) \hookrightarrow \mathcal{O}(E,C)$ will induce an isomorphism on non-stable $K$-theory as well. \medskip \\
Following \cite{Ara}, a ring $R$ (possibly non-unital) is called an \textit{exchange ring} if for any $x \in R$, there exists an idempotent $e \in R$ and elements $r,s \in R$ such that $e=xr$ and $e = s+x-xs$. For the class of $C^*$-algebras, this property coincides with the (to $C^*$-algebraists more familiar) notion of \textit{real rank zero} \cite[Theorem 3.8]{Ara}. The Fundamental Separativity Problem for exchange rings (see \cite{Goodearl}, \cite{AGMP})) asks whether every exchange ring $R$ is \textit{separative}, that is whether the cancellation property
$$2a = a+b = 2b \Rightarrow a=b$$
holds in the non-stable $K$-theory $\mathcal{V}(R)$. A positive answer to this problem would provide positive answers to a number of open problems in both ring-theory and operator algebras \cite[Sections 6 and 7]{AGMP}. In light of the highly general non-stable $K$-theory of separated graph algebras, it is therefore natural to ask when the tame algebras $L_K^{\textup{ab}}(E,C)$, $\mathcal{O}(E,C)$ and $\mathcal{O}^r(E,C)$ are exchange rings. The main result of the present paper provides a somewhat discouraging, but not unexpected, answer to this question: If any one of these is an exchange ring, then it will be a classical graph algebra in case $(E,C)$ is finite and approximately a classical graph algebra in case $(E,C)$ is only finitely separated. In particular, it will be separative. On our way to proving this result, we also take a minor detour (Section~\ref{sect:Simplicity}) to obtain a characterisation of simplicity for these algebras, and the conclusion is similar to the one above: Any of the algebras $L_K(E,C)$, $L_K^{\textup{ab}}(E,C)$, $C^*(E,C)$ and $\mathcal{O}(E,C)$ can only be simple if it is in fact a classical graph algebra, while $\mathcal{O}^r(E,C)$ may also be Morita equivalent to the reduced group $C^*$-algebra $C_r^*(\mathbb{F}_n)$ for $n \ge 2$. This result (although only in the setting of finite separated graphs) is also obtained by Ara and the author in \cite{AL}, but with quite different arguments. \medskip \\
Our proofs of the above mentioned results rest heavily on a dynamical description of the tame algebras $L_K^{\textup{ab}}(E,C)$, $\mathcal{O}(E,C)$ and $\mathcal{O}^r(E,C)$. This was first obtained for finite bipartite graphs in \cite{AE} by Ara and Exel, and we generalise this description to finitely separated graphs in Section 2. In Section 3, we investigate when $L_K(E,C)$ and $C^*(E,C)$ degenerate to graph algebras of non-separated graphs, and we combine the results of Section 2 and 3 to characterise simplicity of the various algebras in Section 4. We then study degeneracy of the tame algebras $L_K^{\textup{ab}}(E,C)$, $\mathcal{O}(E,C)$ and $\mathcal{O}^r(E,C)$ in Section 5, before characterising the exchange property of the various algebras as well as essential freeness of the associated partial action in terms of a graph-theoretic Condition (K) in Section 6.

\section{Preliminary definitions}
In this section, we recall the necessary definitions and results from the existing theory on algebras associated with separated graphs.
\begin{definition}
A \textit{separated graph} $(E,C)$ is a graph $E=(E^0,E^1,r,s)$ together with a \textit{separation} $C=\bigsqcup_{v \in E^0} C_v$, where each $C_v$ is a partition of $r^{-1}(v)$ into non-empty subsets. If $v$ is a source, i.e.~if $r^{-1}(v) = \emptyset$, then we simply take $C_v$ to be the empty partition, and for any $e \in E^1$, the set in $C_{r(e)}$ containing $e$ will be denoted $[e]$. As soon as we start building various objects out of separated graphs, we will only consider \textit{finitely separated} ones, meaning that every $X \in C$ is finite. Finally, any directed graph $E$ may be regarded as a separated graph by giving it the \textit{trivial separation} $\mathcal{T}_v=\{r^{-1}(v)\}$ for all $v \in E^0 \setminus E_{\textup{source}}^0$. Note that in this case, finitely separated simply means column-finite. \medskip \\
A \textit{complete} subgraph $(F,D)$ of $(E,C)$ is a subgraph such that $D_v=\{X \in C_v \mid X \cap F^1 \ne \emptyset \}$ for every $v \in F^0$. The inclusion of a complete subgraph defines a morphism in the category of finitely separated graphs (see \cite[Definition 8.4]{AG2} for details on the appropriate notion of morphism), and every finitely separated graph is the direct limit of its finite complete subgraphs \cite[Proposition 3.5 and Definition 8.4]{AG2}. Moreover, as all the constructions from separated graphs are functorial and preserve direct limits, it is often sufficient to consider only finite graphs.
\end{definition}
\begin{definition}
Let $K$ denote any field. The \textit{Leavitt path algebra} $L_K(E,C)$ associated with a finitely separated graph $(E,C)$ is the $*$-algebra (over $K$) generated by $E^0 \sqcup E^1$ with relations
\begin{enumerate}[leftmargin=2cm,rightmargin=2cm]
\item[(V)] $uv = \delta_{u,v}v$ and $u=u^*$ for $u,v \in E^0$,
\item[(E)] $es(e)=r(e)e=e$ for $e \in E^1$,
\item[(SCK1)] $e^*f = \delta_{e,f}s(e)$ if $[e]=[f]$,
\item[(SCK2)] $v = \sum_{e \in X}ee^*$ for all $v \in E^0$ and $X \in C_v$,
\end{enumerate}
and the graph $C^*$-algebra $C^*(E,C)$ is the universal $C^*$-algebra with respect to these generators and relations. In other words, $C^*(E,C)$ is the enveloping $C^*$-algebra of $L_\mathbb{C}(E,C)$. The reader should be aware that we use the convention of \cite{AE}, \cite{AE2}, \cite{AL} and \cite{Lolk2}, often referred to as the \textit{Raeburn-convention}, opposite to that of \cite{AG} and \cite{AG2} -- consequently, paths will have to be read from the right. \exend
\end{definition}

Note that if $C=\mathcal{T}$, then (SCK1) and (SCK2) are simply the ordinary (CK1) and (CK2) axioms for classical graph algebras, so $L_K(E,C)=L_K(E)$ and $C^*(E,\mathcal{T})=C^*(E)$. It was observed in \cite[Proposition 3.6]{AG2} and \cite[Proposition 1.6]{AG} that the assignments $(E,C) \mapsto L_K(E,C)$ and $(E,C) \mapsto C^*(E,C)$ extend to continuous functors from the category of finitely separated graphs. \medskip \\
Recall that a set of partial isometries $S$ is set to be \textit{tame} if all products of elements from $S \cup S^*$ are also partial isometries.

\begin{definition}[{\cite[Definition 2.4]{AE2}}]
Let $(E,C)$ denote a finitely separated graph. Then $L_K^{\textup{ab}}(E,C)$ is the $*$-algebra (over $K$) and $\mathcal{O}(E,C)$ is the universal $C^*$-algebra generated by $E^0 \sqcup E^1$ with relations (V), (E), (SCK1), (SCK2), and $E^1$ being tame. We refer to $L_K^{\textup{ab}}(E,C)$ as the \textit{abelianised Leavitt path algebra} of $(E,C)$ and to $\mathcal{O}(E,C)$ as \textit{universal tame graph $C^*$-algebra} of $(E,C)$. Since we invoke more relations than above, there are canonical quotient maps
$$L_K(E,C) \to L_K^{\textup{ab}}(E,C) \andspace C^*(E,C) \to \mathcal{O}(E,C).$$
It was proven in \cite[Proposition 7.2]{AE2} that the assignment $(E,C) \mapsto \mathcal{O}(E,C)$ extends to a continuous functor, and the same proof applies to $(E,C) \mapsto L_K^{\textup{ab}}(E,C)$.
\end{definition}

\begin{definition}
A separated graph $(E,C)$ is called \textit{bipartite} if there exists a partition of the vertex set $E^0 = E^{0,0} \sqcup E^{0,1}$ with $s(E^1) = E^{0,1}$ and $r(E^1) = E^{0,0}$. \exend
\end{definition}

Whenever $(E,C)$ is a separated graph, there is a canonical bipartite replacement $\mathbf{B}(E,C)$ as defined in \cite[Proposition 9.1]{AE} and \cite[Definition 7.4]{AE2}. By \cite[Proposition 7.5]{AE2}, the assignment $(E,C) \mapsto \mathbf{B}(E,C)$ is functorial and there are natural isomorphisms of functors 
$$\textbf{M}_2 \circ C^* \cong C^* \circ \mathbf{B} \andspace \textbf{M}_2 \circ \mathcal{O} \cong \mathcal{O} \circ \mathbf{B},$$
where $\textbf{M}_2$ is the functor taking $A$ to $M_2(A)$, and the same proof applies to $L_K$ and $L_K^{\textup{ab}}$. Therefore, one can often restrict to finite bipartite graphs, which was exactly the setup of \cite{AE}. Moreover, if $(E,C)$ is finite and bipartite, then there is a sequence of finite, bipartite separated graphs $(E_n,C^n)$  as defined in \cite[Construction 4.4]{AE}, such that $L_K^{\textup{ab}}(E,C) \cong \varinjlim_n L_K(E_n,C^n)$ and $\mathcal{O}(E,C) \cong \varinjlim_n C^*(E_n,C^n)$ for appropriate connecting $*$-homomorphisms. We have $(E,C)=(E_0,C^0)$, and every $(E_{n+1},C^{n+1})$ is constructed in the same way from $(E_n,C^n)$, so all the graphs $(E_n,C^n)$ give rise to the same tame algebras. \medskip \\
While only finite bipartite separated graphs were considered in \cite{AE}, one can associate a partial action $\theta^{(E,C)} \colon \mathbb{F} \act \Omega(E,C)$ to any finitely separated graph $(E,C)$. Here $\mathbb{F}$ is the free group generated by the edge set $E^1$, and $\Omega(E,C)$ is a zero-dimensional, locally compact and metrisable space (see Definition~\ref{def:DynamicalPicture}). We will verify in Theorem~\ref{thm:DynDescr} that \cite[Corollary 6.12]{AE} generalises, i.e.~that
$$L_K^{\textup{ab}}(E,C) \cong C_K(\Omega(E,C)) \rtimes \mathbb{F} \andspace \mathcal{O}(E,C) \cong C_0(\Omega(E,C)) \rtimes \mathbb{F}$$
for all finitely separated graphs $(E,C)$, where $C_K(\Omega(E,C))$ is the $*$-algebra of locally constant, compactly supported functions $\Omega(E,C) \to K$. This allows for the definition of a \textit{reduced} tame $C^*$-algebra as defined in \cite[Definition 6.8]{AE2}.

\begin{definition}
If $(E,C)$ is a finitely separated graph, then \textit{the reduced tame $C^*$-algebra} associated with $(E,C)$ is the reduced crossed product $\mathcal{O}^r(E,C) := C_0(\Omega(E,C)) \rtimes_{r} \mathbb{F}$.
\end{definition}

\begin{definition}
Let $E$ denote a graph. A non-trivial path in $E$ is a finite, non-empty sequence of edges $\alpha=e_ne_{n-1} \cdots e_2e_1$ satisfying $r(e_i)=s(e_{i+1})$ for all $i=1,\ldots,n-1$. A \textit{subpath} of $\alpha$ is simply a subsequence $e_me_{m-1} \cdots e_{j+1}e_j$, and the range and source of $\alpha$ is defined by $r(\alpha):=r(e_n)$ and $s(\alpha):=s(e_1)$. We will sometimes use the notation $\alpha \colon u \to v$ for a path with source $u$ and range $v$, and we shall regard the vertices $E^0$ as the set of \textit{trivial paths} with $r(v):=v=:s(v)$. \medskip \\
The \textit{double} $\hat{E}$ of $E$ is the graph obtained from $E$ by adding an edge $e^{-1}$ going in the reverse direction for any $e \in E^1$. Namely, $\hat{E}$ is the graph with vertices $\hat{E}^0:=E^0$ and edges 
$$\hat{E}^1:=E^1 \sqcup \{e^{-1} \mid e \in E^1\},$$
where $r$ and $s$ are extended from $E^1$ by $r(e^{-1}):=s(e)$ and $s(e^{-1}):=r(e)$. The map 
$$E^0 \sqcup E^1 \to E^0 \sqcup \hat{E}^1 \quad \text{given by} \quad v \mapsto v \quad\text{and}\quad e \mapsto e^{-1}$$ then extends canonically to an order-reversing, order two bijection of the paths of $\widehat{E}$, denoted by $\alpha \mapsto \alpha^{-1}$. \medskip \\
Now if $(E,C)$ is a separated graph, an \textit{admissible path} $\alpha$ in $(E,C)$ is a path in the double $\hat{E}$, such that
\begin{enumerate}
\item any subpath $e^{-1}f$ with $e,f \in E^1$ satisfies $[e] \ne [f]$,
\item any subpath $ef^{-1}$ with $e,f \in E^1$ satisfies $e \ne f$,
\end{enumerate}
and the set of admissible paths of $(E,C)$ is denoted $\mathcal{P}(E,C)$. If $\alpha$ is a non-trivial admissible path, we shall use the notation $\ini_\textup{d}(\alpha)$ and $\ter_{\textup{d}}(\alpha)$ for the initial (i.e.~rightmost) and terminal (i.e.~leftmost) symbol of $\alpha$, respectively; for instance 
$$\ini_{\textup{d}}(ef^{-1})=f^{-1} \andspace \ter_{\textup{d}}(ef^{-1})=e.$$
Letting $\pi \colon \hat{E}^1 \to E^1$ denote the projection given by $\pi(e):=e=:\pi(e^{-1})$, we then set 
$$\ini(\alpha):=\pi(\ini_{\textup{d}}(\alpha)) \quad \text{as well as} \quad \ter(\alpha):=\pi(\ter_{\textup{d}}(\alpha)).$$
If $\alpha$ is an admissible path and $X,Y \in C$, we shall say that $X^{-1}\alpha$, respectively, $X^{-1}\alpha Y$ is an \textit{admissible composition} if $x^{-1}\alpha$, respectively, $x^{-1}\alpha y$ is admissible for some (hence any) $x \in X$ and $y \in Y$. We finally introduce a partial order $\le$ on $\mathcal{P}(E,C)$ given by
$$\beta \le \alpha \Leftrightarrow \text{$\beta$ is an initial subpath of $\alpha$}.$$
In particular, whenever $s(\alpha)=s(\beta)$, there is a maximal initial subpath $\alpha \wedge \beta \le \alpha, \beta$. \exend
\end{definition}
We shall use the exact same terminology as above when inverse edges $e^{-1}$ are replaced by adjoint edges $e^*$.

\begin{notation}
Given admissible paths $\alpha$ and $\beta$ in $(E,C)$, we will write $\beta \alpha$ for the concatenated product, which may or may not be an admissible path. However, one may also view $\alpha$ and $\beta$ as elements of the free group $\mathbb{F}$ on $E^1$, in which a product can also be formed, allowing for cancellation of edges and their inverses. To distinguish notationally between these two products, we will always write $\beta \cdot \alpha$ when the product is formed in $\mathbb{F}$.
\end{notation}

\begin{remark}\label{rem:StandardForm}
Note that in light of the defining relations of $C^*(E,C)$, any non-zero product $\alpha \in C^*(E,C)$ of elements from $E^1 \sqcup (E^1)^*$ can be written as
$$\alpha=\alpha_n\alpha_{n-1} \cdots \alpha_2\alpha_1,$$
where each $\alpha_i$ is a non-trivial admissible path, and
\begin{enumerate}
\item $\ini_{\textup{d}}(\alpha_{i+1}) \in E^1$ and $\ter_{\textup{d}}(\alpha_i)=\ini_{\textup{d}}(\alpha_{i+1})^*$,
\item $\vert [\ini_{\textup{d}}(\alpha_{i+1})] \vert \ge 2$
\end{enumerate}
for all $i=1,\ldots,n-1$.
\end{remark}

\begin{definition}
A non trivial admissible path $\alpha$ in a separated graph is called a \textit{closed path} if $r(\alpha)=s(\alpha)$, and it is called a \textit{cycle} if $\alpha\alpha$ is an admissible path as well. In either case, we shall say that $\alpha$ is \textit{based} at $r(\alpha)=s(\alpha)$. An admissible path $\alpha$ will be referred to as \textit{simple} if it does not meet the same vertex twice, i.e.~if $r(\alpha_1)=r(\alpha_2)$ for $\alpha_1,\alpha_2 \le \alpha$ implies $\alpha_1=\alpha_2$, while a cycle is called a \textit{simple cycle} if the only vertex repetition occurs at the end, that is if $\alpha_1 < \alpha_2 \le \alpha$ and $r(\alpha_1)=r(\alpha_2)$ implies $\alpha_1=s(\alpha)$ and $\alpha_2=\alpha$. Any closed path $\alpha$ is called \textit{base-simple} if $s(\alpha)$ is only repeated at the end. \exend
\end{definition}

As for non-separated graphs, there is a notion of \textit{hereditary} and \textit{saturated} sets, giving rise to ideals in our algebras.

\begin{definition}[{\cite[Definition 6.3 and Definition 6.5]{AG2}}]
Let $(E,C)$ denote a finitely separated graph. A set of vertices $H \subset E^0$ is called \textit{hereditary} if $r(e) \in H$ implies $s(e) \in H$ for all $e \in E^1$, and it is called $C$\textit{-saturated} if for all $v \in E^0$ and $X \in C_v$, $s(X) \subset H$ implies $v \in H$. We will write $\mathcal{H}(E,C)$ for the lattice of hereditary and $C$-saturated sets. Finally, for any such $H \in \mathcal{H}(E,C)$, we define a quotient graph $(E/H,C/H)$ by 
\begin{itemize}
\item $E/H)^0:= E^0 \setminus H$ and $(E/H)^1:=r^{-1}(H)$,
\item $(C/H)_v:=\{X/H \mid X \in C_v\}$ where $X/H:=X \cap r^{-1}(H)$,
\end{itemize}
with the range and source restricted from $E$.
\end{definition}

\section{Dynamical systems associated with finitely separated graphs}

In this section, we generalise Ara and Exel's dynamical description of the tame algebras $L_K^{\textup{ab}}(E,C)$ and $\mathcal{O}(E,C)$ in \cite{AE} to finitely separated graphs, allowing us to define $\mathcal{O}^r(E,C)$ for such graphs. We also address the relationship between the partial actions $\theta^{(E,C)}$ and $\theta^{\mathbf{B}(E,C)}$ associated to a finitely separated graph and its bipartite sibling, respectively. First, however, we will recall the basics of partial actions.

\begin{definition}
A \textit{partial action} $\theta \colon G \act \Omega$ of a discrete group $G$ on a topological space $\Omega$ is a family of homeomorphisms of open subspaces $\{\theta_g \colon \Omega_{g^{-1}} \to \Omega_g\}_{g \in G}$, such that
\begin{itemize}
\item $\theta_g(\Omega_{g^{-1}} \cap \Omega_h) \subset \Omega_{gh}$ for all $g,h \in G$,
\item $\theta_g(\theta_h(x))=\theta_{gh}(x)$ for all $g,h \in G$ and $x \in \Omega_{h^{-1}} \cap \Omega_{h^{-1}g^{-1}}$,
\end{itemize}
and we will always assume $\Omega$ to be locally compact Hausdorff. Completely similarly, one can define the concept of a partial action on a $(C)^*$-algebra, demanding that the domains should be (closed) ideals. Hence, $\theta$ as above translates into a partial $C^*$-action $\theta^* \colon G \act C_0(\Omega)$ given by $C_0(\Omega)_g := C_0(\Omega_g)$ and $\theta_g^*(f):=f \circ \theta_g^{-1}$
for all $g \in G$. As for global actions, one can associated both a \textit{full} and a \textit{reduced} crossed product, and there is a canonical surjective $*$-homomorphism
$$C_0(\Omega) \rtimes G \to C_0(\Omega) \rtimes_r G,$$
called the \textit{regular representation}. We will often write $\rtimes_{(r)}$ to indicate that a given statement concerns both crossed products. If the space $\Omega$ is totally disconnected and $K$ is any field with involution, there is also a meaningful, purely algebraic partial action $\theta^* \colon G \act C_K(\Omega)$ on the $*$-algebra $C_K(\Omega)$ of compactly supported, locally constant functions, and this gives rise to a single algebraic crossed product $C_K(\Omega) \rtimes G$. We refer the reader to \cite{Exel} for a comprehensive treatment of crossed products associated with partial actions. \medskip \\
Returning to the topological setting, a subspace $U \subset \Omega$ is called \textit{invariant} if $\theta_g(x) \in U$ for all $g \in G$ and $x \in U \cap \Omega_{g^{-1}}$. Observe that whenever $U$ is open and invariant, then $Z:=\Omega \setminus U$ is closed and invariant, so $\theta$ naturally restricts to partial actions of both $U$ and $Z$, giving rise to sequences
$$0 \to C_K(U) \rtimes G \to C_K(\Omega) \rtimes G \to C_K(Z) \rtimes G \to 0$$
and
$$0 \to C_0(U) \rtimes_{(r)} G \to C_0(\Omega) \rtimes_{(r)} G \to C_0(Z) \rtimes_{(r)} G \to 0.$$
On the level of full crossed products and in the purely algebraic context, this sequence is always exact, but for reduced crossed products, one must also require the group to be exact \cite[Theorem 22.9]{Exel}; as we are really only interested in free groups, this is not a problem. The \textit{orbit} through any $x \in \Omega$ is the set
$$\theta_G(x) := \{\theta_g(x) \mid g \in G \text{ such that } x \in \Omega_{g^{-1}}\},$$
and the action is called \textit{minimal} if every orbit is dense in $\Omega$, or equivalently if the only open invariant subspaces are the trivial ones. The action is called \textit{topologically free} if for every $1 \ne g \in G$, the set of fixed points
$$\Omega^g:=\{x \in \Omega_{g^{-1}} \mid \theta_g(x)=x\} \subset \Omega$$
has empty interior. Minimality is of course a necessary condition for simplicity of any of the above crossed products, and if the action is topologically free, then it is also sufficient in both the algebraic and reduced context (see \cite[Remark 3.9]{AL}, \cite[Lemma 3.1 and Theorem 4.1]{BCFS} and \cite[Corollary 29.8]{Exel}). A partial action is called \textit{essentially free} if the restriction to every closed invariant subset is topologically free. Essential freeness (together with exactness of the group in the reduced setting) guarantees that all ideals of the algebraic and reduced crossed product are induced from open invariant subspaces (see \cite[Corollary 3.7]{CEHS} and \cite[Theorem 29.9]{Exel}). \medskip \\
Now if $U \subset \Omega$ is any open subspace (not necessarily invariant), we may still define a restricted partial action $\theta\vert_U \colon G \act U$ with domains $U_g:=\theta_g(U \cap \Omega_{g^{-1}}) \cap U$. Following \cite{CRS} and \cite{Li2}, such a restriction will be called \textit{full} (or \textit{$G$-full}) if $\Omega=\{\theta_g(x) \mid g \in G, x \in U \cap \Omega_{g^{-1}}\}$. \exend \end{definition}

We now recall a number of different types of equivalences between partial actions.

\begin{definition}
Suppose that $\theta \colon G \act \Omega$ and $\theta' \colon H \act \Omega'$ are partial actions and $\Phi \colon G \to H$ is a group homomorphism. A continuous map $\varphi \colon \Omega \to \Omega'$ is called $\Phi$\textit{-equivariant} if
\begin{itemize}
\item $\varphi(\Omega_g) \subset \Omega'_{\Phi(g)}$ for all $g \in G$,
\item $\theta'_{\Phi(g)}(\varphi(x))=\varphi(\theta_g(x))$ for all $g \in G$ and $x \in \Omega_{g^{-1}}$.
\end{itemize}
If $G=H$ and $\Phi=\id_G$, then $\varphi$ is simply called equivariant (or possibly $G$-equivariant). The pair $(\varphi,\Phi)$ is called a \textit{conjugacy} if $\varphi$ is a homeomorphism, $\Phi$ is an isomorphism, and $\varphi ^{-1}$ is $\Phi^{-1}$-equivariant. However, conjugacy is often too rigid a notion and we therefore consider a few other types of equivalences: Following \cite{AL}, the pair $(\varphi,\Phi)$ is called a \textit{direct dynamical equivalence} if
\begin{enumerate}[(a)]
\item $\varphi$ is a homeomorphism,
\item $\Omega_g \cap \Omega_{g'} = \emptyset$ for all $g \ne g'$ with $\Phi(g)=\Phi(g')$,
\item $\Omega'_h=\bigcup_{\Phi(g)=h} \varphi(\Omega_g)$ for all $h \in H$.
\end{enumerate} 
If, moreover, $\Phi$ is injective, the pair $(\varphi,\Phi)$ will be called a \textit{direct quasi-conjugacy}. \textit{Dynamical equivalence} and \textit{quasi-conjugacy} are then simply the equivalence relations on partial actions generated by these two non-symmetric relations. It should be obvious that any dynamical property is preserved by dynamical equivalence. In fact, by \cite[Proposition 3.11 and Proposition 3.13]{AL}, dynamical equivalence is exactly the same as isomorphism of the transformation groupoids $\mathcal{G}_\theta$ and $\mathcal{G}_{\theta'}$. Finally, borrowing from \cite{Li2} and \cite{CRS}, $\theta$ and ${\theta'}$ are called \textit{Kakutani equivalent} if there exist full, clopen subspaces $K \subset \Omega$ and $K' \subset \Omega'$ such that the restrictions $\theta\vert_K$ and ${\theta'}\vert_{K'}$ are dynamically equivalent. As was noted just below \cite[Definition 3.24]{AL}, Kakutani equivalence implies Morita equivalence of the associated crossed products.  \exend

\end{definition}

The following results about full subspaces and Kakutani equivalence will be useful later.

\begin{lemma}\label{lem:FullSubset}
Suppose that $\theta \colon G \act \Omega$ is a partial action on a locally compact Hausdorff space, and let $U \subset \Omega$ denote an open, full subspace. Then there is a bijective correspondence
$$\{\text{open $\theta$-invariant subsets of } \Omega\} \to \{\text{open $\theta\vert_U$-invariant subsets of } U\}  \text{ given by } V \mapsto U \cap V.$$
Moreover, for any open and $\theta$-invariant $V \subset \Omega$, the following hold:
\begin{enumerate}
\item[\textup{(1)}] $U \cap V$ is full in $V$.
\item[\textup{(2)}] $U \cap (\Omega \setminus V)$ is full in $\Omega \setminus V$.
\end{enumerate}
\end{lemma}

\begin{proof}
Any intersection $U \cap V$ clearly defines an open $\theta\vert_U$-invariant subset of $U$, so we simply have to build an inverse. Suppose that $W \subset U$ is open and $\theta\vert_U$-invariant, and define an open and $\theta$-invariant subset of $\Omega$ by 
$$V:=\bigcup_{g \in G}\theta_g(W \cap \Omega_{g^{-1}}).$$
Observe that if $x \in W \cap \Omega_{g^{-1}}$, then either $\theta_g(x) \notin U$ or $\theta_g(x) \in W$, hence
$$U \cap V= \bigcup_{g \in G} U \cap \theta_g(W \cap \Omega_{g^{-1}}) = W.$$
Now whenever $y \in V$, then by fullness of $U$, we have $y = \theta_h(x)$ for some $h \in G$ and $x \in U \cap \Omega_{h^{-1}}$. From invariance, we see that $x \in V$, so $y \in \theta_h(U \cap V \cap \Omega_{h^{-1}})$. Consequently
$$\bigcup_{g \in G} \theta_g(U \cap V \cap \Omega_{g^{-1}}) = V,$$
so the above map is indeed a bijective correspondence, and (1) holds. (2) then follows immediately from (1) and fullness of $U$.
\end{proof}

\begin{lemma}\label{lem:KETopFree}
Topological freeness is preserved under Kakutani equivalence.
\end{lemma}

\begin{proof}
Obviously, topological freeness is preserved under direct dynamical equivalence, so we simply have to show that a partial action $\theta$ is topologically free if and only if a $G$-full restriction $\theta\vert_K$ to a clopen subspace $K \subset \Omega$ is topologically free. Assume the latter and take some $x \in \Omega^g$ along with an open neighbourhood $U \subset \Omega_{g^{-1}}$ of $x$. By fullness, there exists some $h \in G$ and $y \in K \cap \Omega_{h^{-1}}$ such that $x=\theta_h(y)$. We then define an open neighbourhood of $y$ in $K$ by
$$V:= \theta_{h^{-1}}(U \cap \theta_h(K \cap \Omega_{h^{-1}}))$$
and observe that $y \in K^{h^{-1}gh}$. Now since $\theta\vert_K$ is topologically free, we can find $y' \in V \cap K_{h^{-1}g^{-1}h}$ such that $\theta_{h^{-1}gh}(y') \ne y'$. Setting $x':=\theta_h(y') \in U$, we then have
$$x'=\theta_h(y') \ne \theta_h(\theta_{h^{-1}gh}(y'))=\theta_{gh}(y')=\theta_g(x')$$
as desired. The other implication is trivial. \end{proof}

\begin{corollary}\label{cor:KEEssFree}
Essential freeness is preserved under Kakutani equivalence.
\end{corollary}

\begin{proof}
Once again, it suffices to verify the claim for a partial action $\theta \colon G \act \Omega$ and the restriction $\theta\vert_K$ to a clopen, full subset $K \subset \Omega$. But then the claim follows immediately from Lemma~\ref{lem:FullSubset} and Lemma~\ref{lem:KETopFree}.
\end{proof}

We now describe the partial action giving rise to a crossed product description of $\mathcal{O}(E,C)$.

\begin{definition}\label{def:DynamicalPicture}
Suppose that $(E,C)$ is a finitely separated graph, and let $\mathbb{F}$ denote the free group on $E^1$. Also, denote by $E_{\textup{iso}}^0$ the set of isolated vertices with the discrete topology. Given $\xi \subset \mathbb{F}$ and $\alpha \in \xi$, the \textit{local configuration} $\xi_\alpha$ of $\xi$ at $\alpha$ is the set
$$\xi_\alpha:=\{\sigma \in E^1 \sqcup (E^1)^{-1} \mid \sigma \in \xi \cdot \alpha^{-1}\}.$$
Then $\Omega(E,C)$ is the disjoint union of $E_{\textup{iso}}^0$ and the set of $\xi \subset \mathbb{F}$ satisfying the following:
\begin{enumerate}
\item[(a)] $1 \in \xi$.
\item[(b)] $\xi$ is \textit{right-convex}. In view of (a), this exactly means that if $e_n^{\varepsilon_n} \cdots e_1^{\varepsilon_1} \in \xi$ for $e_i \in E^1$ and $\varepsilon_i \in \{\pm 1\}$, then $e_m^{\varepsilon_m} \cdots e_1^{\varepsilon_1} \in \xi$ as well for any $1 \le m < n$.
\item[(c)] For every $\alpha \in \xi$, there is some $v \in E^0$ and distinguished $e_X \in X$ for each $X \in C_v$, such that 
$$\xi_\alpha = s^{-1}(v) \sqcup \{e_X^{-1} \mid X \in C_v \}.$$
\end{enumerate}
$\Omega(E,C)$ is made into a topological space by regarding it as a subspace of $\{0,1\}^\mathbb{F} \sqcup E_{\textup{iso}}^0$. Thus it becomes a zero-dimensional, locally compact Hausdorff space, which is compact if and only if $E^0$ is a finite set. A topological partial action $\theta=\theta^{(E,C)} \colon \mathbb{F} \act \Omega(E,C)$ with compact-open domains is then defined by setting
\begin{itemize}
\item $\Omega(E,C)_\alpha := \{\xi \in \Omega(E,C) \setminus E_{\textup{iso}}^0 \mid \alpha^{-1} \in \xi\}$ for $\alpha \ne 1$,
\item $\theta_\alpha(\xi) := \xi \cdot \alpha^{-1}$ for $\xi \in \Omega(E,C)_{\alpha^{-1}}$.
\end{itemize}
In case $(E,C)$ is a finite bipartite graph, this partial action is conjugate to the one defined in \cite{AE} under the map $\xi \mapsto \xi^{-1}$. We choose to invert the configurations so that the terminologies related to the algebras and the dynamical systems agree. We set
$\Omega(E,C)_{s(e)}:= \Omega(E,C)_{e^{-1}}$ for every $e \in E^1$ and
$\Omega(E,C)_u := \bigsqcup_{e \in X}\Omega(E,C)_e$ for every  $X \in C_u$. Note that this is well-defined due to the above condition (c). If $u$ is an isolated vertex, we simply set $\Omega(E,C)_u:=\{u\}$. Finally, in the case of a trivial separation, we will write $\Omega(E):=\Omega(E,\mathcal{T})$ and $\theta^E := \theta^{(E,\mathcal{T})}$.
\end{definition}

\begin{remark}
The partial action $\theta^E$ is easily seen to be conjugate to the canonical partial action of $\mathbb{F}$ on the boundary path space $\partial E$ (see \cite{CL} and adjust the definition to the Raeburn-convention). However, there are also graphs with non-trivial separations that give rise to boundary path space actions. Indeed, Proposition~\ref{prop:NonSepGraphIso} and Proposition~\ref{prop:GraphCondForOr} together identifies a class of such graphs, where the identification is made by an actual conjugacy. Relaxing conjugacy to dynamical equivalence and considering only finite bipartite graphs, we further strengthen this result with Theorem~\ref{thm:ODege}.
\end{remark}

\begin{remark}\label{rem:CompleteSubgraphMap}
If $(F,D)$ is a complete subgraph of $(E,C)$, then there is a natural $\mathbb{F}(F^1)$-equivariant surjection $p \colon \bigsqcup_{v \in F^0} \Omega(E,C)_v \to \Omega(F,D)$ given by
$$p(\xi) = \left\{\begin{array}{cl}
\xi \cap \mathbb{F}(F^1) &  \If \xi \in \Omega(E,C)_v \text{ for } v \notin F_{\textup{iso}}^0 \\
v & \If \xi \in \Omega(E,C)_v \text{ for } v \in F_{\textup{iso}}^0
\end{array}\right. .$$
Consequently, it induces $\mathbb{F}(F^1)$-equivariant embeddings
$$C_K(\Omega(F,D)) \xrightarrow{p^*} C_K\big(\bigsqcup_{v \in F^0} \Omega(E,C)_v \big) \hookrightarrow C_K(\Omega(E,C))$$
and
$$C_0(\Omega(F,D)) \xrightarrow{p^*} C_0\big(\bigsqcup_{v \in F^0} \Omega(E,C)_v \big) \hookrightarrow C_0(\Omega(E,C))$$
from which we obtain $*$-homomorphisms
$$C_K(\Omega(F,D)) \rtimes \mathbb{F}(F^1) \to C_K(\Omega(E,C)) \rtimes \mathbb{F}(F^1) \to C_K(\Omega(E,C)) \rtimes \mathbb{F}(E^1)$$
and
$$C_0(\Omega(F,D)) \rtimes_{(r)} \mathbb{F}(F^1) \to C_0(\Omega(E,C)) \rtimes_{(r)} \mathbb{F}(F^1) \to C_0(\Omega(E,C)) \rtimes_{(r)} \mathbb{F}(E^1).$$
Finally, observe that taking the limits over the finite complete subgraphs with inclusions, we have 
$$C_K(\Omega(E,C)) \cong \varinjlim_{(F,D)} C_K(\Omega(F,D)) \andspace C_0(\Omega(E,C)) \cong \varinjlim_{(F,D)} C(\Omega(F,D))$$
for any finitely separated graph $(E,C)$, and if $E_{\textup{iso}}^0 = \emptyset$, we have the same approximations when considering only finite complete subgraphs $(F,D)$ with $F_{\textup{iso}}^0=\emptyset$. \exend
\end{remark}

We introduce a bit of terminology related to the closed subspaces of $\Omega(E,C)$, while the definition of $\Omega(E,C)$ is still fresh in mind.

\begin{definition}\label{def:Animal}
An $(E,C)$\textit{-animal} is a right-convex subset $\omega \subset \xi$ of a configuration $\xi \in \Omega(E,C) \setminus E_{\textup{iso}}^0$ such that $\{1\} \subsetneq \omega$. It is called finite if it has finite cardinality, and for any animal $\omega$, we can define a compact subset of $\Omega(E,C)$ by
$$\Omega(E,C)_\omega := \{\xi \in \Omega(E,C) \mid \omega \subset \xi\},$$
which is open if $\omega$ is finite. Given any non-empty subset $\{1\} \ne S \subset \mathbb{F}$ such that $\alpha \cdot \beta^{-1}$ is an admissible path for any pair of distinct $\alpha,\beta \in S \cup \{1\}$, observe that the right-convex closure $\langle S \rangle:=\text{conv}(S \cup \{1\})$ of $S \cup \{1\}$ inside $\mathbb{F}$ defines an $(E,C)$-animal. In order to avoid confusion, the reader should also note that we have the slightly annoying identity $\Omega(E,C)_\alpha=\Omega(E,C)_{\{\alpha^{-1}\}}$. \medskip \\
The \textit{balls} are a particularly important type of animals: An $n$-ball is simply a set of the form $\xi^n:=\{ \alpha \in \xi \colon \vert \alpha \vert \le n\}$ together with the radius $n$ (we sometimes want to distinguish balls with the same underlying set and different radii). If $(E,C)$ is finite, then any finite animal is contained in a ball, and the compact-open subsets $\Omega(E,C)_B$ corresponding to the balls $B$ form a basis for the topology. We will denote the set of $n$-balls by $\mathcal{B}_n(\Omega(E,C))$. \exend
\end{definition}

We now prove that $\theta^{(E,C)}$ does in fact provide a dynamical description of $L_K^{\textup{ab}}(E,C)$ and $\mathcal{O}(E,C)$. While the original proof from \cite{AE} proceeds in a constructive manner, applying the machinery of \cite{ELQ} to translate the defining relations into restrictions on the local configurations $\xi_\alpha$ for $\xi \in \mathcal{P}(\mathbb{F})$ with $1 \in \xi$, we will aim for a more direct and conceptually easier, but also somewhat unmotivated proof.

\begin{theorem}\label{thm:DynDescr}
For any finitely separated graph $(E,C)$, there are canonical isomorphisms
$$L_K^{\textup{ab}}(E,C) \cong C_K(\Omega(E,C)) \rtimes \mathbb{F} \andspace \mathcal{O}(E,C) \cong C_0(\Omega(E,C)) \rtimes \mathbb{F}.$$
\end{theorem}

\begin{proof}
We may assume without loss of generality that $E_{\textup{iso}}^0 = \emptyset$. Denote by $1_\alpha$ and $1_v$ the indicator function on $\Omega(E,C)_\alpha=\{\xi \in \Omega(E,C) \mid \alpha^{-1} \in \xi\}$ (remember the slightly confusing inversion) and $\Omega(E,C)_v$, respectively, and write $u_\alpha:=1_\alpha \delta_\alpha$ for all $\alpha \in \mathbb{F}$. We then consider the elements $u_e$ and $p_v:=1_v$ in $C_K(\Omega(E,C)) \rtimes \mathbb{F}$ for $e \in E^1$ and $v \in E^0$, claiming that they form a tame $(E,C)$-family. (V) and (E) are both clear, while (SCK1) follows from the calculation
$$u_e^*u_f=u_{e^{-1}}u_f=\theta_{e^{-1}}^*(1_e1_f)u_{e^{-1}f}=\delta_{e,f} 1_{e^{-1}} = \delta_{e,f}p_{s(e)}$$
whenever $[e]=[f]$. Noting that $u_eu_e^*=u_eu_{e^{-1}}=1_e$ for $e \in E^1$, we also see that 
$$\sum_{e \in X} u_eu_e^*=\sum_{e \in X} 1_e=p_v$$
for any $v \in E^0$ and $X \in C_v$, so that (SCK2) is satisfied. Now let $\alpha$ denote any reduced product of edges and inverse edges $\alpha=\alpha_n \cdots \alpha_1$. We then introduce the notation $\underline{e}:=e$, $\underline{e^{-1}}:=e^*$ and $\underline{\alpha}:=\underline{\alpha_n} \cdots \underline{\alpha_1}$, and claim that 
$$u_{\underline{\alpha}}:=u_{\underline{\alpha_n}} \cdots u_{\underline{\alpha_1}} = u_\alpha.$$
Assuming the claim holds for products of length $n-1$ and writing $\beta=\alpha_{n-1} \cdots \alpha_1$, we see that
$$u_{\underline{\alpha}}=u_{\underline{\alpha_n}}u_{\underline{\beta}}=1_{\alpha_n} 1_{\alpha} u_\alpha=u_\alpha,$$
where we used right-convexity to conclude that $1_{\alpha_n} 1_{\alpha}=1_{\alpha}$. It follows that 
$$u_{\underline{\alpha}}u_{\underline{\alpha}}^*=u_\alpha u_\alpha^*=1_\alpha,$$
so in particular the set $\{u_e \mid e \in E^1\}$ is tame. From universality, we therefore obtain a $*$-homomorphism $\varphi \colon L_K^{\textup{ab}}(E,C) \to C_K(\Omega(E,C)) \rtimes \mathbb{F}$ given by $e \mapsto u_e$ and $v \mapsto p_v$.  \medskip \\
We now begin the construction of an inverse by first building a $*$-homomorphism 
$$\rho \colon C_K(\Omega(E,C)) \to L_K^{\textup{ab}}(E,C).$$
To this end, let $(F,D)$ denote any finite complete subgraph with $F_{\textup{iso}}^0=\emptyset$, let $n \ge 1$ and write 
$$1_B:=\prod_{\alpha \in B} 1_{\alpha^{-1}} \in C_K(\Omega(F,D))$$
for any $B \in \mathcal{B}_n(\Omega(F,D))$; $1_B$ is merely the indicator function on the subspace $\Omega(F,D)_B$. We then define finite-dimensional subalgebras
$$\mathfrak{B}_n^{(F,D)}:=\text{span}\{1_B \mid B \in \mathcal{B}_n(\Omega(F,D))\} \subset C_K(\Omega(F,D))$$
along with inclusions $\phi_n^{(F,D)} \colon \mathfrak{B}_n^{(F,D)} \to \mathfrak{B}_{n+1}^{(F,D)}$ given by
$$\phi_n^{(F,D)}(1_B)=\sum_{B \subset B' \in \mathcal{B}_{n+1}(\Omega(F,D))} 1_{B'},$$
and observe that $C_K(\Omega(F,D)) = \varinjlim_n \mathfrak{B}_n^{(F,D)}$. Now if $(F,D) \subset (G,L)$, then the inclusion of Remark~\ref{rem:CompleteSubgraphMap} restricts to an inclusion $\mathfrak{B}_n^{(F,D)} \hookrightarrow \mathfrak{B}_n^{(G,L)}$, which makes the diagram

\begin{center}
\begin{tikzpicture}[>=angle 90]
\matrix(a)[matrix of math nodes,
row sep=2.5em, column sep=3.5em,
text height=2.5ex, text depth=0.25ex]
{\mathfrak{B}_n^{(F,D)} & \mathfrak{B}_{n+1}^{(F,D)} \\
\mathfrak{B}_n^{(G,L)} & \mathfrak{B}_{n+1}^{(G,L)} \\};
\path[->](a-1-1) edge node[above]{$\phi_n^{(F,D)}$} (a-1-2);
\path[->](a-2-1) edge node[below]{$\phi_n^{(G,L)}$} (a-2-2);
\path[right hook->](a-1-1) edge node[left]{} (a-2-1);
\path[right hook->](a-1-2) edge node[right]{} (a-2-2);
\end{tikzpicture}
\end{center}
commute. In conclusion, defining a $*$-homomorphism out of $C_K(\Omega(E,C))$ is the same as defining a family of $*$-homomorphisms out of the algebras $\mathfrak{B}_n^{(F,D)}$ that respects both the horizontal and vertical maps above. Now consider the self-adjoint linear map 
$$\rho_n^{(F,D)} \colon \mathfrak{B}_n^{(F,D)} \to L_K^{\textup{ab}}(E,C) \quad \text{given by} \quad \rho_n^{(F,D)}(1_B)=\prod_{\alpha \in B} \underline{\alpha}^*\underline{\alpha},$$
which is well-defined since $E^1 \subset L_K^{\textup{ab}}(E,C)$ is tame (see for instance \cite[Proposition 12.8]{Exel}). Checking that $\rho_n^{(F,D)}$ is also multiplicative exactly amounts to showing 
$$\rho_n^{(F,D)}(1_{B_1})\rho_n^{(F,D)}(1_{B_2})=0$$
for all $B_1 \ne B_2$. Since
$$\rho_n^{(F,D)}(1_B) \le \rho_m^{(F,D)}\big(1_{B^m}\big),$$
where $B^m:=\{\alpha \in B \colon \vert \alpha \vert \le m\}$, for all $m \le n$, we may assume that $B_1^{n-1} = B_2^{n-1}$. Consequently there is some $\beta \in B_1,B_2$ of length $\vert \beta \vert = n-1$ and $X \in C_{r(\beta)}$ with distinct $x_1,x_2 \in X$ such that $\alpha_1 := x_1^{-1}\beta \in B_1$ and $\alpha_2:=x_2^{-1}\beta \in B_2$. We see that 
$$\underline{\alpha_1}^*\underline{\alpha_1}\hspace{1pt}\underline{\alpha_2}^*\underline{\alpha_2} = x_1\underline{\beta}^*\underline{\beta}x_1^*x_2\underline{\beta}^*\underline{\beta}x_2=0,$$
so $\rho_n^{(F,D)}(1_{B_1})\rho_n^{(F,D)}(1_{B_2})=0$ as well. In order to see that these $*$-homomorphisms respect both the vertical and horizontal inclusions above, simply note that both follow from an inductive application of the following observations. Given any finite $(E,C)$-animal $\omega$, the following hold:
\begin{enumerate}
\item If $1 \ne \beta \in \omega$ and $\omega_\beta \cap X^{-1} = \emptyset$ for some $X \in C_{r(\beta)}$, then
$$\prod_{\alpha \in \omega} \underline{\alpha}^*\underline{\alpha}=\sum_{x \in X}\prod_{\alpha \in \omega \cup \{x^{-1}\beta\}} \underline{\alpha}^*\underline{\alpha}.$$
\item If $1 \ne \beta \in \omega$ and $e \notin \omega_\beta$ for some $e \in s^{-1}(r(\beta))$, then
$$\prod_{\alpha \in \omega}\underline{\alpha}^*\underline{\alpha} = \prod_{\alpha \in \omega \cup \{e \beta\}}\underline{\alpha}^*\underline{\alpha}.$$
\end{enumerate}

We thereby obtain a unique $*$-homomorphism $\rho \colon C_K(\Omega(E,C)) \to L_K^{\textup{ab}}(E,C)$ characterised by $\rho(1_\alpha)=\underline{\alpha}\hspace{1pt}\underline{\alpha}^*$ for any $\alpha \in \mathbb{F}$. Now observe that 
$$\varphi \circ \rho(1_\alpha)=\varphi(\underline{\alpha}\hspace{1pt}\underline{\alpha}^*)=u_{\underline{\alpha}}u_{\underline{\alpha}}^*=1_\alpha$$
for any $\alpha$, so the composition $\varphi \circ \rho$ is nothing but the inclusion $C_K(\Omega(E,C)) \hookrightarrow C_K(\Omega(E,C)) \rtimes \mathbb{F}$. Since $E^1$ is tame in $L_K^{\textup{ab}}(E,C)$, by the implication (iii)$\Rightarrow$(i) of \cite[Proposition 12.13]{Exel} which holds in an arbitrary unital $*$-algebra, there is a semi-saturated partial representation $\sigma$ of $\mathbb{F}$ on the unitalisation of $L_K^{\textup{ab}}(E,C)$ given by $\sigma(\alpha):=\underline{\alpha}$ for all $\alpha \ne 1$, so that $\rho(1_\alpha)=p(\alpha):=\sigma(\alpha)\sigma(\alpha)^*$. We claim that the pair $(\rho,\sigma)$ is a covariant representation. It suffices to check that 
$$\sigma(\alpha)\rho(1_{\alpha^{-1}} 1_\beta)\sigma(\alpha)^*=\rho(\theta_\alpha^*(1_{\alpha^{-1}}1_\beta))$$
for all $\alpha,\beta \in \mathbb{F}$, and from \cite[Proposition 9.8(iii)]{Exel}, we have $p(\beta)\sigma(\alpha)^*=\sigma(\alpha)^*p(\alpha\beta)$. We now see that
\begin{align*}
\sigma(\alpha)\rho(1_{\alpha^{-1}}1_\beta)\sigma(\alpha)^* &= \sigma(\alpha)p(\alpha^{-1})p(\beta)\sigma(\alpha)^*= \sigma(\alpha)p(\beta)\sigma(\alpha)^* = \sigma(\alpha)\sigma(\alpha)^* p(\alpha \cdot\beta) \\
&=p(\alpha)p(\alpha \cdot \beta)=\rho(1_\alpha 1_{\alpha \cdot \beta})=\rho(\theta_\alpha(1_{\alpha^{-1}}1_{\beta}))
\end{align*}
as desired, so there is an induced $*$-homomorphism $\rho \times \sigma \colon C_K(\Omega(E,C)) \rtimes \mathbb{F} \to L_K^{\textup{ab}}(E,C)$. Since
$$(\rho \times \sigma) \circ \varphi(e)=\rho \times \sigma(1_e\delta_e)=ee^*e=e$$
for all $e \in E^1$, we have $(\rho \times \sigma) \circ \varphi=\id$. Moreover, the fact that $\varphi \circ \rho$ is simply the inclusion $C_K(\Omega(E,C)) \hookrightarrow C_K(\Omega(E,C)) \rtimes \mathbb{F}$ together with the observation
$$\varphi \circ (\rho \times \sigma)(u_\alpha)=\varphi(\underline{\alpha}\hspace{1pt}\underline{\alpha}^*\underline{\alpha})=\varphi(\underline{\alpha})=u_{\underline{\alpha}}=u_\alpha$$
for all $\alpha \in \mathbb{F}$ implies that $\varphi \circ (\rho \times \sigma) = \id$ as well. It follows that $L_K^{\textup{ab}}(E,C) \cong C_K(\Omega(E,C)) \rtimes \mathbb{F}$ as desired, and the $C^*$-case is completely similar.
\end{proof}

We now see that the $*$-homomorphisms coming from inclusions of complete subgraphs are simply those of Remark~\ref{rem:CompleteSubgraphMap}.

\begin{lemma}\label{lem:MorphDesc}
Let $(E,C)$ denote a finitely separated graph, and consider an embedding of a complete subgraph $(F,D) \stackrel{\iota}{\hookrightarrow} (E,C)$. Then the $*$-homomorphisms 
$$L_K^{\textup{ab}}(\iota) \colon L_K^{\textup{ab}}(F,D) \to L_K^{\textup{ab}}(E,C) \andspace \mathcal{O}(\iota) \colon \mathcal{O}(F,D) \to \mathcal{O}(E,C)$$
are exactly the ones of Remark~\ref{rem:CompleteSubgraphMap}. Consequently, there is a unique $*$-homomorphism 
$$\mathcal{O}^r(\iota) \colon \mathcal{O}^r(F,D) \to \mathcal{O}^r(E,C)$$
making the diagram
\begin{center}
\begin{tikzpicture}[>=angle 90]
\matrix(a)[matrix of math nodes,
row sep=2.5em, column sep=3.5em,
text height=2.5ex, text depth=0.25ex]
{\mathcal{O}(F,D) & \mathcal{O}(E,C) \\
\mathcal{O}^r(F,D) & \mathcal{O}^r(E,C) \\};
\path[->](a-1-1) edge node[above]{$\mathcal{O}(\iota)$} (a-1-2);
\path[->](a-2-1) edge node[below]{$\mathcal{O}^r(\iota)$} (a-2-2);
\path[->>](a-1-1) edge node[left]{} (a-2-1);
\path[->>](a-1-2) edge node[right]{} (a-2-2);
\end{tikzpicture}
\end{center}
commute.
\end{lemma}

\begin{proof}
Simply observe that $L_K^{\textup{ab}}(\iota)$ and $\mathcal{O}(\iota)$ agree with the homomorphisms of Remark~\ref{rem:CompleteSubgraphMap} on the generators.
\end{proof}

We can also characterise $\theta^{(E,C)}$ by a very useful universal property, corresponding to the universal property of $\mathcal{O}(E,C)$. Recall that whenever $\{\theta_a \mid a \in A\}$ is a family of homeomorphisms of open subspaces of a space $\Omega$, there is a canonical partial action of $\mathbb{F}(A)$ on $\Omega$: If $\alpha=a_n^{\varepsilon_n} \cdots a_1^{\varepsilon_1}$ is a reduced word, then $\theta_\alpha:= \theta_{a_n}^{\varepsilon_n} \cdots \theta_{a_1}^{\varepsilon_1}$, where $\cdot$ denotes the maximal composition of two functions. It is verified in \cite[Proposition 4.7]{Exel} that this does indeed define a partial action.

\begin{definition}
Suppose that $\Omega$ is a locally compact Hausdorff space and $(E,C)$ is a finitely separated graph. An $(E,C)$\textit{-action} on $\Omega$ is the canonical action of $\mathbb{F}=\mathbb{F}(E^1)$ induced by a family $\{\theta_e \colon \Omega_{e^{-1} } \to \Omega_e \mid e \in E^1\}$ of partial homeomorphisms of compact open subspaces with the following properties:
\begin{enumerate}
\item There is a decomposition $\Omega=\bigsqcup_{v \in E^0} \Omega_v$ for compact open subspaces $\Omega_v \subset \Omega$.
\item If $e \in E^1$, then $\Omega_{s(e)}=\Omega_{e^{-1}}$.
\item If $v \in E^0$ and $X \in C_v$, then $\Omega_v = \bigsqcup_{e \in X} \Omega_e$.
\end{enumerate}
An $(E,C)$-action $\theta \colon \mathbb{F} \act \Omega$ is called \textit{universal} if, given any other $(E,C)$-action $\theta' \colon \mathbb{F} \act \Omega'$, there exists a unique $\mathbb{F}$-equivariant continuous map $f \colon \Omega' \to \Omega$ such that $f(\Omega_v') \subset f(\Omega_v)$ for any isolated vertex $v$ (this will automatically hold for any other vertex due to equivariance). Observe that a universal $(E,C)$-action is unique up to canonical conjugacy. Finally, if $C=\mathcal{T}$, we will simply suppress the separation, referring instead to an $E$-action. \exend
\end{definition}

The following can be obtained from Theorem~\ref{thm:DynDescr} by applying duality, but we choose to a give a concrete proof for clarity.

\begin{proposition}\label{prop:UnivPropTop}
The partial action $\theta^{(E,C)}$ is the universal $(E,C)$-action for any finitely separated graph $(E,C)$. If $\theta \colon \mathbb{F} \act \Omega$ is any other $(E,C)$-action, then the unique equivariant map $f \colon \Omega \to \Omega(E,C)$ satisfying $p(\Omega_v) \subset \Omega(E,C)_v$ for $v \in E_{\textup{iso}}^0$ is given by
$$p(x)=\left\{ \begin{array}{cl}
\mathbb{F}^x := \{\alpha \in \mathbb{F} \mid x \in \Omega_{\alpha^{-1}}\} & \If x \in \bigsqcup_{v \in E^0 \setminus E_{\textup{iso}}^0} \Omega_v \\
v & \If x \in \Omega_v \textup{ for } v \in E_{\textup{iso}}^0
\end{array}\right. . $$
\end{proposition}
\begin{proof}
It is clear that $\theta^{(E,C)}$ is itself an $(E,C)$-action, and that the restriction of $p$ to $\bigsqcup_{v \in E_{\textup{iso}}^0} \Omega_v$ is the unique map satisfying $p(\Omega_v) \subset \Omega(E,C)_v$. Considering any $x \in \bigsqcup_{v \in E^0 \setminus E_{\textup{iso}}^0} \Omega_v$, it is also clear that $\mathbb{F}^x$ is a right-convex set containing $1$. We have $x \in \Omega_v$ for a unique $v \in E^0$, so $x \in \Omega_{e^{-1}}$ if and only if $e \in s^{-1}(v)$, and $x \in \Omega_{e_X}$ for a unique $e_X \in X$ for all $X \in C_v$. Hence the local configuration of $\mathbb{F}^x$ at $1$ is given by
$$s^{-1}(v) \sqcup \{e_X^{-1} \mid X \in C_v\}$$
as in Definition~\ref{def:DynamicalPicture}(c). Now if $\alpha \in \mathbb{F}^x$, then we may apply this observation to $\theta_\alpha(x)$ to see that $\mathbb{F}^x_\alpha=\mathbb{F}^{\theta_\alpha(x)}_1$ is of the same type, thus $\mathbb{F}^x \in \Omega(E,C)$. Equivariance and continuity of $x \mapsto \mathbb{F}^x$ is obvious, and if $\varphi \colon \Omega \to \Omega(E,C)$ is any equivariant map, then necessarily $\varphi(x)=\mathbb{F}^{\varphi(x)} \supset \mathbb{F}^x$. Now since $\mathbb{F}^x,\mathbb{F}^{\varphi(x)} \in \Omega(E,C)$, we must have $\mathbb{F}^{\varphi(x)}=\mathbb{F}^x$, and so $\varphi(x)=\mathbb{F}^x=p(x)$.
\end{proof}

\begin{remark}
Note that if $(F,D)$ is a complete subgraph of $(E,C)$, and $\theta$ is the restriction of $\theta^{(E,C)}$ to $\bigsqcup_{v \in F^0} \Omega(E,C)_v$, then the map $p$ from Remark~\ref{rem:CompleteSubgraphMap} is exactly the map $p$ of Proposition~\ref{prop:UnivPropTop}. \exend
\end{remark}

Now that we have a dynamical system associated to every finitely separated graph, we will shortly consider the relationship between the dynamics of $(E,C)$ and its bipartite sibling $\mathbf{B}(E,C)$ as defined in \cite[Definition 7.4]{AE2}. First though, we have to introduce a bit of terminology.

\begin{definition}\label{def:Double}
For any topological $\Omega$, we will write $\overrightarrow{\Omega}=\Omega^1 \sqcup \Omega^0$, where each 
$$\Omega^i = \{\xi^i \mid \xi \in \Omega\}$$
is simply a copy of $\Omega$. Given a partial homeomorphism $\varphi$ of $\Omega$, we define  a partial homeomorphism $\overrightarrow{\varphi} \colon \text{dom}(\varphi)^1 \to \text{im}(\varphi)^0$ by $\overrightarrow{\varphi}(\xi^1)=\varphi(\xi)^0$. Now if $\theta \colon \mathbb{F}(A) \act \Omega$ is a partial action induced from a family of partial homeomorphisms $\{\theta_a\}_{a \in A}$, we define the \textit{double action} of $\theta$ to be the partial action $\overrightarrow{\theta} \colon \mathbb{F}(A) \ast \mathbb{Z} \act \overrightarrow{\Omega}$ induced by the family $\{\overrightarrow{\theta_a}\}_{a \in A}$ and $\sigma:=\overrightarrow{\id_\Omega}$.
\end{definition}

\begin{proposition}\label{prop:Double}
Consider a partial action $\theta$ as in Definition~\ref{def:Double}. For any $i=0,1$, there is a direct quasi-conjugacy $\theta \to \overrightarrow{\theta}\vert_{\Omega^i}$ and a direct dynamical equivalence $\overrightarrow{\theta}\vert_{\Omega_i} \to \theta$.
\end{proposition}

\begin{proof}
We only consider the case $i=1$; the other one is completely analogous. Denote the generator of $\mathbb{Z}$ by $s$ so that $\mathbb{F}(A) \ast \mathbb{Z}=\mathbb{F}(A \cup \{s\})$, and consider the injective homomorphism $\Phi \colon \mathbb{F}(A) \to \mathbb{F}(A \cup \{s\})$ given by $\Phi(a)=s^{-1}a$ for all $a \in A$ as well as the embedding $\varphi \colon \Omega \to \overrightarrow{\Omega}$ onto $\Omega^1$. Then
$${\big(\overrightarrow{\theta}\big)}_{\Phi(a)}={\big(\overrightarrow{\theta}\big)}_{s^{-1}a} = \sigma^{-1} \circ \overrightarrow{\theta_a}= \varphi \circ \theta_a$$
for all $a \in A$, so $(\varphi,\Phi)$ is a conjugacy of $\theta$ and the restricted partial action $\text{im}(\Phi) \act \Omega^1$. It simply remains to check that $\Omega^1_\beta = \emptyset$ for all $\beta \in \mathbb{F}(A \cup \{s\}) \setminus \text{im}(\Phi)$. Observe that such $\beta$, as a reduced word, must contain a subword either of one of the forms $sa$, $as$, $aa'$ for $a,a' \in A$ or an inverse of one of these. In every case, we see that $\Omega^1_\beta = \emptyset$, as desired. Since $\Phi^{-1}$ can be extended to a group homomorphism $\Psi \colon \mathbb{F}(A \cup \{s\}) \to \mathbb{F}(A)$, namely $\Psi(a)=a$ for $a \in A$ and $\Psi(s)=1$, we see that $(\varphi^{-1},\Psi)$ defines a direct dynamical equivalence $\overrightarrow{\theta}\vert_{\Omega^1} \to \theta$.
\end{proof}

We now relate the partial actions of $(E,C)$ and $\mathbf{B}(E,C)$. 

\begin{proposition}\label{prop:BipDynamics}
Let $(E,C)$ denote a finitely separated graph and write $(\tilde{E},\tilde{C}):=\mathbf{B}(E,C)$ as well as $\theta:=\theta^{(E,C)}$. Then there is a direct dynamical equivalence $\theta^{(\tilde{E},\tilde{C})} \to \overrightarrow{\theta}$. In particular, $\theta^{(\tilde{E},\tilde{C})}$ and $\theta$ are Kakutani equivalent.
\end{proposition}

\begin{proof}
Recall that $(\tilde{E},\tilde{C})$ is the finitely separated bipartite graph given by
\begin{itemize}
\item $\tilde{E}^{0,i}:=\{v_i \mid v \in E^0\}$ for $i=0,1$,
\item $\tilde{E}^1:=\{\tilde{e} \mid e \in E^1\} \cup \{h_v \mid v \in E^0\}$,
\item $\tilde{r}(\tilde{e}):=r(e)_0$ and $\tilde{s}(\tilde{e}):=s(e)_1$ for all $e \in E^1$,
\item $\tilde{r}(h_v):=v_0$ and $\tilde{s}(h_v):=v_1$ for all $v \in E^0$,
\item $\tilde{C}_{v_0}:=\{\{h_v\}, \tilde{X} \mid X \in C_v\}$ where $\tilde{X}:=\{\tilde{e} \mid e \in X\}$ for all $v \in E^0$.
\end{itemize}
As above, we denote the generator of the factor $\mathbb{Z}$ by $s$, and we will write $\Omega:=\Omega(E,C)$. We first define an $(\tilde{E},\tilde{C})$-action $\gamma$ on $\overrightarrow{\Omega}$ by
\begin{itemize}
\item ${\big(\overrightarrow{\Omega}\big)}_{v_i}:=\Omega_v^i$ for all $v \in E^0$ and $i=0,1$,
\item ${\big(\overrightarrow{\Omega}\big)}_{\tilde{e}^{-1}}:=\Omega_{e^{-1}}^1$, ${\big(\overrightarrow{\Omega}\big)}_{\tilde{e}}:=\Omega_e^0$ and $\gamma_{\tilde{e}}:=\overrightarrow{\theta_e}$ for all $e \in E^1$,
\item ${\big(\overrightarrow{\Omega}\big)}_{h_v^{-1}}:=\Omega_v^1$, ${\big(\overrightarrow{\Omega}\big)}_{h_v}:=\Omega_v^0$ and $\gamma_{h_v}:=\overrightarrow{\id_{\Omega_v}}=\sigma\vert_{\Omega_v^1}$ for all $v \in E^0$.
\end{itemize}
Observe that there is a direct dynamical equivalence $(\id,\Phi)\colon \gamma \to \overrightarrow{\theta}$, where $\Phi(e)=e$ and $\Phi(h_v)=s$ for all $e \in E^1$ and $v \in E^0$. Now, by the universal property of $\theta^{(\tilde{E},\tilde{C})}$, there is a unique $\mathbb{F}(\tilde{E}^1)$-equivariant continuous map $\varphi \colon \overrightarrow{\Omega} \to \Omega(\tilde{E},\tilde{C})$, and we claim that this is in fact a conjugacy. To see this, we first define injective group homomorphisms $\Psi^1,\Psi^0 \colon \mathbb{F}(E^1) \to \mathbb{F}(\tilde{E}^1)$ by 
$$\Psi^1(e)=h_{r(e)}^{-1}\tilde{e} \andspace \Psi^0(e)=\tilde{e}h_{s(e)}^{-1},$$
and observe (just as in Proposition~\ref{prop:Double}) that the identification $\Omega \cong \Omega^i$ together with $\Psi^i$ defines a direct quasi-conjugacy $\theta \to \gamma\vert_{\Omega^i}$. Next, define an $(E,C)$-action $\gamma^i$ on 
$$\Omega^i(\tilde{E},\tilde{C}):=\bigsqcup_{v \in E^0} \Omega(\tilde{E},\tilde{C})_{v_i}$$
for $i=0,1$ by
\begin{itemize}
\item $\Omega^i(\tilde{E},\tilde{C})_v:=\Omega(\tilde{E},\tilde{C})_{v_i}$ for all $v \in E^0$,
\item $\Omega^i(\tilde{E},\tilde{C})_{e^{\pm 1}}:=\Omega(\tilde{E},\tilde{C})_{\Psi^i(e^{\pm 1})}$ and $\gamma^i_e:=\theta^{(\tilde{E},\tilde{C})}_{\Psi^i(e)}$ for all $e \in E^1$.
\end{itemize}
From the universal property of $(E,C)$ and the observations just above, there is a unique $\mathbb{F}(E^1)$-equivariant continuous map $\psi^i \colon \Omega^i(\tilde{E},\tilde{C}) \to \Omega^i$, and $\psi^i \circ \varphi\vert_{\Omega^i} = \id_{\Omega^i}$ by uniqueness. Setting $\psi: = \psi^1 \sqcup \psi^0 \colon \Omega(\tilde{E},\tilde{C}) \to \Omega$ so that $\psi \circ \varphi = \id_{\overrightarrow{\Omega}}$, we claim that $\psi$ is in fact $\mathbb{F}(\tilde{E}^1)$-equivariant. By construction, it is equivariant under both $\text{im}(\Psi^1)$ and $\text{im}(\Psi^0)$, so we simply have to check that it is also equivariant under the action of every $h_v$, i.e.~that the diagram

\begin{center}
\begin{tikzpicture}[>=angle 90]
\matrix(a)[matrix of math nodes,
row sep=4em, column sep=6em,
text height=1.5ex, text depth=0.25ex]
{\Omega^1(\tilde{E},\tilde{C}) & \Omega^0(\tilde{E},\tilde{C}) \\ 
\Omega^1 & \Omega^0 \\};
\path[->](a-1-1) edge node[above]{$\bigsqcup_{v \in E^0} \theta^{(\tilde{E},\tilde{C})}_{h_v}$} (a-1-2);
\path[->](a-2-1) edge node[below]{$\sigma$} (a-2-2);
\path[->](a-1-1) edge node[left]{$\psi_1$} (a-2-1);
\path[->](a-1-2) edge node[right]{$\psi_0$} (a-2-2);
\end{tikzpicture}
\end{center}
commutes. Note that all four entries carry partial actions of $\mathbb{F}(E^1)$, and that the maps are all equivariant with respect to these actions. Since the action of $\mathbb{F}(E^1)$ on $\Omega^0$ is the universal $(E,C)$-action, uniqueness of $\mathbb{F}(E^1)$-equivariant maps $\Omega^1(\tilde{E},\tilde{C}) \to \Omega^0$ guarantees that the diagram actually commutes. We conclude that $\varphi \circ \psi$ is $\mathbb{F}(\tilde{E}^1)$-equivariant, hence $\varphi \circ \psi = \id_{\Omega(\tilde{E},\tilde{C})}$ as desired.
\end{proof}

We finally observe that hereditary and $C$-saturated subsets, just as for finite bipartite separated graphs, give rise to ideals in the tame algebras. When $\xi \in \Omega(E,C)_v$ is a configuration, we regard $1 \in \xi$ as the trivial path $v$ and so $r(1):=v$ by convention.

\begin{definition}
Given a hereditary and $C$-saturated subset $H \subset E^0$, we define
$$\Omega(E,C)^H:=\{\xi \in \Omega(E,C) \mid r(\alpha) \in H \text{ for some } \alpha \in \xi\}.$$
\end{definition}

\begin{theorem}\label{thm:QuoByHS}
Let $(E,C)$ denote a finitely separated graph, and consider a hereditary and $C$-saturated set of vertices $H \subset E^0$. Then $\Omega(E,C)^H$ is an open and invariant subspace, and there is a direct quasi-conjugacy $\theta^{(E/H,C/H)} \to \theta^{(E,C)}\vert_Z$ where $Z:=\Omega(E,C) \setminus \Omega(E,C)^H$. Letting $I(H)$ denote the induced ideal in the various algebras, which is exactly the ideal generated by $H$, we therefore have isomorphisms
$$L_K^{\textup{ab}}(E,C)/I(H) \cong L_K^{\textup{ab}}(E/H,C/H) \andspace \mathcal{O}^{(r)}(E,C)/I(H) \cong \mathcal{O}^{(r)}(E/H,C/H).$$
\end{theorem}
\begin{proof}
Simply observe that the proof of \cite[Theorem 5.5]{AL} (or rather the second part of it) generalises with minimal effort.
\end{proof}

\section{Degeneracy of $L_K(E,C)$ and $C^*(E,C)$}\label{sect:GraphAlgIso}

In this section, we give a sufficient condition for $L_K(E,C)$ and $C^*(E,C)$ to be isomorphic to a graph algebra of a non-separated graph; we regard this as a degenerating situation since these algebras are well studied. This isomorphism is always implemented by reversing certain edges of the separated graph, a technique also used by Duncan in \cite{Duncan2}. The concepts and theorems of this section will be used heavily in subsequent sections on simplicity and the exchange property.  \medskip \\
First we will need to introduce the following essential definitions.
\begin{definition}[{\cite[Definition 9.5]{AL}}]
Let $(E,C)$ denote a finitely separated graph. An admissible path $\alpha$ is called a \textit{choice path} if there is an admissible composition $X^{-1}\alpha$ for some $X \in C$ with $\vert X \vert \ge 2$.
\end{definition}

\begin{definition}
Let $(E,C)$ denote a finitely separated graph. We shall say that $e \in E^1$ \textit{admits a choice} if there is a choice path $\alpha$ satisfying $\ini_{\textup{d}}(\alpha)=e$, while an inverse edge $e^{-1}$ admits a choice if $\vert [e] \vert \ge 2$, or if there is a choice path $\alpha$ with $\ini_{\textup{d}}(\alpha)=e^{-1}$. A set $X \in C$ then admits a choice if $e^{-1}$ admits a choice for some $e \in X$, and finally a vertex $v \in E^0$ is said to admit exactly
$$\vert \{e \in s^{-1}(v) \mid e \text{ admits a choice}\} \vert + \vert \{X \in C_v \mid X \text{ admits a choice}\}\vert$$
choices. \exend
\end{definition}

The important distinction -- as we will see -- is between those vertices that admit no, those that admit exactly one, and those that admit at least two choices. The following easy lemma guarantees that the equivalence relation of being on the same cycle respects this distinction.

\begin{lemma}\label{lem:NumberOfChoices}
Let $(E,C)$ denote a finitely separated graph. If $u$ and $v$ are on the same cycle, then
\begin{enumerate}
\item[\textup{(1)}] $u$ admits no choices if and only $v$ admits no choices,
\item[\textup{(2)}] $u$ admits exactly one choice if and only $v$ admits exactly one choice,
\item[\textup{(3)}] $u$ admits at least two choices if and only if $v$ admits at least two choices.
\end{enumerate}
\end{lemma}
\begin{proof}
Say that $\beta\alpha$ is a cycle with $s(\alpha)=u$ and $r(\alpha)=v$. Observe that if $\sigma \in r^{-1}(u)^{-1} \cup s^{-1}(u)$ admits a choice at $u$, then so does either $\ini_{\textup{d}}(\beta)$ or $\ter_{\textup{d}}(\alpha)^{-1}$, as either $\sigma\beta$ or $\sigma \alpha^{-1}$ is admissible. Now assume that $\sigma,\tau \in r^{-1}(u)^{-1} \cup s^{-1}(u)$ give rise to two different choices at $u$. If $\sigma\beta$ is not admissible, then both $\tau\beta$ and $\sigma\alpha^{-1}$ must be admissible, hence $v$ admits at least two choices as well. Likewise we may assume, without loss of generality, that $\tau\beta$ is admissible, thereby verifying (3). (2) now follows automatically.
\end{proof}

\begin{definition}
If $\alpha$ is a cycle passing through $v$, then we will say that $\alpha$ admits
\begin{enumerate}
\item no choices, if $v$ admits no choices,
\item exactly one choice, if $v$ admits exactly one choice,
\item at least two choices, if $v$ admits at least two choices.
\end{enumerate}
By Lemma~\ref{lem:NumberOfChoices}, this is independent of the choice of $v$ on $\alpha$.
\end{definition}

\begin{definition}
Let $(E,C)$ denote a finitely separated graph. We will say that $(E,C)$ satisfies Condition (C) if every $v \in E^0$ admits at most one choice. \exend
\end{definition}

Recall that a non-separated graph $E$ is said to satisfy Condition (L) if for every cycle $\alpha=e_n \cdots e_1$, there is some $1 \le k \le n$ and $f \ne e_k$ with $r(e_k)=r(f)$. The edge $f$ is usually referred to as an \textit{entry} of $\alpha$. It is well known that $C^*(E)$ is simple if and only if $E$ satisfies Condition (L) and has only trivial hereditary and saturated subsets. Ara and Exel defined Condition (L) for finite bipartite separated graphs in \cite{AE}, and as it will play an important role in the next few sections, we now redefine it in the language of this paper for arbitrary finitely separated graphs.

\begin{definition}[{\cite[Definition 10.2]{AE}}]\label{def:CondL}
A finitely separated graph $(E,C)$ is said to satisfy Condition (L) if any simple cycle admits a choice.
\end{definition} 

\begin{theorem}[{\cite[Theorem 10.5]{AE}}]\label{thm:CondL}
Let $(E,C)$ denote a finitely separated graph. Then $\theta^{(E,C)}$ is topologically free if and only if $(E,C)$ satisfies Condition \textup{(}L\textup{)}.
\end{theorem}
\begin{proof}
The strategy from \cite[Theorem 10.5]{AE} easily generalises to arbitrary finitely separated graphs.
\end{proof}

\begin{definition}\label{def:NonSepOr}
A \textit{non-separated orientation} of a finitely separated graph $(E,C)$ is a decomposition $E^1= E_-^1 \sqcup E_+^1$ such that $[e]=\{e\}$ for every $e \in E_+^1$ and one of the following holds for any $v \in E^0$:
\begin{enumerate}
\item $E_-^1 \cap r^{-1}(v) \in C_v$ and $E_+^1 \cap s^{-1}(v)=\emptyset$.
\item $E_-^1 \cap r^{-1}(v) = \emptyset$ and $\vert E_+^1 \cap s^{-1}(v) \vert \le 1$.
\end{enumerate}
This is a special case of an \textit{orientation}, which is defined in \cite[Definition 3.11]{Lolk2}. We shall often regard the partition $E=E_-^1 \sqcup E_+^1$ as a map $\mathfrak{o} \colon E^1 \to \{-1,1\}$ with
$$\mathfrak{o}(e):=\left\{ \begin{array}{cl}
1 & \If e \in E_+^1 \\
-1 & \If e \in E_-^1
\end{array}\right. .$$
An admissible path of the form
$$e_n^{\mathfrak{o}(e_n)} e_{n-1}^{\mathfrak{o}(e_{n-1})} \cdots e_2^{\mathfrak{o}(e_2)}e_1^{\mathfrak{o}(e_{1})}$$
will then be called \textit{positively oriented}, while a path of the form
$$e_n^{-\mathfrak{o}(e_n)} e_{n-1}^{-\mathfrak{o}(e_{n-1})} \cdots e_2^{-\mathfrak{o}(e_2)}e_1^{-\mathfrak{o}(e_{1})}$$
will be called \textit{negatively oriented}. By \cite[Lemma 3.12]{Lolk2}, every admissible path $\alpha$ decomposes as $\alpha=\alpha_-\alpha_+$, where $\alpha_+$ and $\alpha_-$ are positively and negatively oriented, respectively. \exend
\end{definition}

The point of a non-separated orientation is that it allows us to turn a separated graph into a non-separated one.

\begin{definition}
Assume that $(E,C)$ is a finitely separated graph. If $(E,C)$ admits a non-separated orientation $E^1=E_-^1 \sqcup E_+^1$, then we can define a corresponding column-finite directed graph $\overline{E}=(\overline{E}^0,\overline{E}^1,\overline{r},\overline{s})$ by $\overline{E}^0=\{ \overline{v} \mid v \in E^0$\}, $\overline{E}^1=\{\overline{e} \mid e \in E^1\}$,
$$\overline{r}(\overline{e})=\left\{ \begin{array}{cl}
r(e) & \text{if } e \in E_-^1 \\
s(e) & \text{if } e \in E_+^1
\end{array}\right. \andspace \overline{s}(\overline{e})=\left\{ \begin{array}{cl}
s(e) & \text{if } e \in E_-^1 \\
r(e) & \text{if } e \in E_+^1
\end{array}\right. .$$
\end{definition}\exend \medskip

Before considering the relationship between the dynamics and algebras of $(E,C)$ and $\overline{E}$, we first record a graph-theoretical lemma for later use.

\begin{lemma}\label{lem:Paths}
Assume that $\mathfrak{o}$ is a non-separated orientation of $(E,C)$, and let $\overline{E}$ denote the resulting non-separated graph. The map
$$\overline{e_n} \hspace{2pt} \overline{e_{n-1}} \cdots \overline{e_2}\hspace{2pt}\overline{e_1} \mapsto e_n^{-\mathfrak{o}(e_n)}e_{n-1}^{-\mathfrak{o}(e_{n-1})} \cdots e_2^{-\mathfrak{o}(e_2)} e_1^{-\mathfrak{o}(e_1)}$$
is a bijection between the paths of $\overline{E}$ and the negatively oriented paths of $(E,C)$.
\end{lemma}
\begin{proof}
First observe that $\overline{r}(\overline{e})=r(e^{-\mathfrak{o}(e)})$ and $\overline{s}(\overline{e})=s(e^{-\mathfrak{o}(e)})$ for all $e \in E^1$, so the above correspondence takes inverse paths of $\overline{E}$ to paths in the double $\hat{E}$, and vice versa. We simply have to check that the paths in the double are also admissible, i.e.~that if $e_i \in E_-^1$ and $e_{i+1} \in E_+^1$, then $[e_i] \ne [e_{i+1}]$. But this is clear since either $X \subset E_-^1$ or $X \subset E_+^1$ for all $X \in C$.
\end{proof}

\begin{proposition}\label{prop:NonSepGraphIso}
Assume that $(E,C)$ is a finitely separated graph with a non-separated orientation, and let $\overline{E}$ denote the resulting non-separated graph. Then $\theta^{(E,C)}$ is conjugate to $\theta^{\overline{E}} \colon \mathbb{F} \act \Omega(\overline{E})$,
$$L_K(E,C) = L_K^{\textup{ab}}(E,C) \cong L(\overline{E}) \andspace C^*(E,C) = \mathcal{O}(E,C) \cong C^*(\overline{E}).$$
Moreover, a subset $H \subset E^0$ is hereditary and $C$-saturated if and only if $\overline{H}=\{\overline{v} \mid v\in H\}$ is hereditary and saturated in $\overline{E}$.
\end{proposition}

\begin{proof}
First, observe that we can define an $(E,C)$-action $\gamma$ on $\Omega(\overline{E})$ by
\begin{itemize}
\item $\Omega(\overline{E})_v := \Omega(\overline{E})_{\overline{v}}$ for $v \in E^0$,
\item $\Omega(\overline{E})_{e^{\pm 1}} := \Omega(\overline{E})_{\overline{e}^{\mp 1}}$ and $\gamma_e := \theta_{\overline{e}^{-1}}^{\overline{E}}$ for $e \in E_+^1$,
\item $\Omega(\overline{E})_{e^{\pm 1}} :=\Omega(\overline{E})_{\overline{e}^{\pm 1}}$ and $\gamma_e:=\theta_{\overline{e}}^{\overline{E}}$ for $e \in E_-^1$,
\end{itemize}

as well as an $\overline{E}$-action $\sigma$ on $\Omega(E,C)$ by
\begin{itemize}
\item $\Omega(E,C)_{\overline{v}}:=\Omega(E,C)_v$ for $v \in E^0$,
\item $\Omega(E,C)_{\overline{e}^{\pm 1}} := \Omega(E,C)_{e^{\mp 1}}$ and $\sigma_{\overline{e}} := \theta^{(E,C)}_{e^{-1}}$ for $e \in E_+^1$,
\item $\Omega(E,C)_{\overline{e}^{\pm 1}} := \Omega(E,C)_{e^{\pm 1}}$ and $\sigma_{\overline{e}} := \theta^{(E,C)}_e$ for $e \in E_-^1$.
\end{itemize}

We then obtain equivariant maps $\Omega(\overline{E}) \to \Omega(E,C)$ and $\Omega(E,C) \to \Omega(\overline{E})$ from the universal properties, and by uniqueness, these must be mutual inverses. Consequently,
$$\mathcal{O}(E,C) \cong C_0(\Omega(E,C)) \rtimes \mathbb{F} \cong C_0(\Omega(\overline{E})) \rtimes \mathbb{F} \cong C^*(\overline{E}).$$
Likewise, one can check that $p_{\overline{v}}=v$ for $v \in E^0$ and
$$t_{\overline{e}}=\left\{\begin{array}{cl}
e & \text{if } e \in E_-^1 \\
e^* & \text{if } e \in E_+^1
\end{array}\right. $$
for $e \in E^1$ defines an $\overline{E}$-family inside $C^*(E,C)$, so that the isomorphism $C^*(\overline{E}) \to \mathcal{O}(E,C)$ factors through $C^*(E,C)$ as a surjection. It follows that $C^*(\overline{E}) \cong C^*(E,C)$ as well. The same argument applies to the Leavitt path algebras. \medskip \\
For the last part of the proposition, suppose that $H \subset E^0$ is hereditary and $C$-saturated with respect to $(E,C)$. In order to check that $\overline{H}$ is hereditary in $\overline{E}$, we assume that $\overline{r}(\overline{e}) \in \overline{H}$. If $e \in E_-^1$, then $\overline{r(e)}=\overline{r}(\overline{e}) \in \overline{H}$ and so $r(e) \in H$. It follows that $\overline{s}(\overline{e})=\overline{s(e)} \in \overline{H}$, so let us instead assume that $e \in E_+^1$. Then $\overline{s(e)}=\overline{r}(\overline{e}) \in \overline{H}$, so $s(e) \in H$. $C$-saturation and $[e]=\{e\}$ imply $r(e) \in H$, hence $\overline{s}(\overline{e})=\overline{r(e)} \in \overline{H}$ as well. To see that $\overline{H}$ is saturated, suppose that $\overline{s}(\overline{r}^{-1}(\overline{v})) \subset \overline{H}$ for some $v \in E^0$ with $\overline{r}^{-1}(\overline{v}) \ne \emptyset$. First assume that $v$ satisfies Definition~\ref{def:NonSepOr}(1). Then $\overline{r}^{-1}(\overline{v})=\{\overline{e} \mid e \in E_-^1 \cap r^{-1}(v)\}$ and 
$$\overline{s}(\overline{r}^{-1}(\overline{v}))=\{\overline{s(e)} \mid e \in E_-^1 \cap r^{-1}(v)\} \subset \overline{H},$$
so $s(E_-^1 \cap r^{-1}(v)) \subset H$. $C$-saturation now implies $v \in H$ as well. Assuming instead that $v$ satisfies (2) of Definition~\ref{def:NonSepOr}, we must have $E_+^1 \cap s^{-1}(v) = \{e\}$ and $\overline{r}^{-1}(\overline{v})=\{\overline{e}\}$ for some $e$, hence $\overline{s}(\overline{r}^{-1}(v)) = \{\overline{r(e)}\} \subset \overline{H}$, i.e.~$r(e) \in H$. We deduce that $v=s(e) \in H$ from $H$ being hereditary, so $v \in \overline{H}$ as desired. The other implication can easily be proven by analogous arguments, and we therefore leave it to the reader.
\end{proof}

While the concept of a non-separated orientation is quite handy for technical purposes, it is certainly not a very natural one. Instead we shall give a sufficient graph-theoretic condition for the existence of such an orientation below in Proposition~\ref{prop:GraphCondForOr}. First though, we need a number of minor technical results to introduce and apply the notion of a \textit{simple closed path}.

\begin{lemma}\label{lem:ClosedPathNoChoice}
If $\alpha$ is a closed path based at a vertex $v$, which admits no choices, then neither does any vertex on $\alpha$.
\end{lemma}
\begin{proof}
This is obvious.
\end{proof}

\begin{lemma}\label{lem:ReducedProduct}
Assume that $v$ admits no choices and that $\alpha, \beta$ are admissible paths with $r(\alpha)=s(\beta)=v$. Then the reduced product $\beta \cdot \alpha$ is an admissible path.
\end{lemma}
\begin{proof}
Set $\gamma:=\alpha^{-1} \wedge \beta$ and write $\alpha=\gamma^{-1} \alpha'$, $\beta=\beta' \gamma$. We then need to verify that the reduced product $\beta \cdot \alpha = \beta' \alpha'$ is admissible. By construction, we must have $r(\alpha')=s(\beta')$ and $\ter_{\textup{d}}(\alpha')^{-1} \ne \ini_{\textup{d}}(\beta')$, so assuming that $\ter_{\textup{d}}(\alpha') \in E^1$ and $\ini_{\textup{d}}(\beta') \in (E^1)^{-1}$, we must simply check that $[\ini_{\textup{d}}(\beta')^{-1}] \ne [\ter_{\textup{d}}(\alpha')]$. But this is evident as $v$ would otherwise admit a choice.
\end{proof}

\begin{proposition}\label{prop:ClosedPathStandardForm}
Assume that a vertex $v \in E^0$ admits no choices. Then every closed path $\alpha$ based at $v$ decomposes uniquely as $\alpha = \gamma^{-1} \beta \gamma$ for a cycle $\beta$, and
$$\mathbb{F}_v:= \{\text{closed paths based at } v\} \cup \{1\}$$
forms a free subgroup $\mathbb{F}_v \le \mathbb{F}$.
\end{proposition}

\begin{proof}
For the first part, set $\gamma:= \alpha \wedge \alpha^{-1}$ and write $\alpha = \gamma^{-1} \beta \gamma$ --  we then claim that $\beta$ is a cycle. If it were not, then we would have $\ini_{\textup{d}}(\beta) \in (E^1)^{-1}$, $\ter_{\textup{d}}(\beta) \in E^1$, $[\ter_{\textup{d}}(\beta)] = [\ini_{\textup{d}}(\beta)^{-1}]$ and $\ter_{\textup{d}}(\beta) \ne \ini_{\textup{d}}(\beta)^{-1}$, hence $r(\beta)$ would admit a choice, contradicting Lemma~\ref{lem:ClosedPathNoChoice}. The second part of the claim is immediate from Lemma~\ref{lem:ReducedProduct} and the Nielsen-Schreier Theorem.
\end{proof}

\begin{definition}
Assume that $v \in E^0$ admits no choices. Then a non-trivial closed path $\alpha$ based at $v$ is called a \textit{simple closed path} if $\alpha=\gamma^{-1} \beta \gamma$ for a simple admissible path $\gamma$ and a simple cycle $\beta$. Obviously, if $\alpha$ is a simple closed path, then so is $\alpha^{-1}$, and so we say that $v$ admits
$$\frac{1}{2} \cdot \big\vert \big\{\text{simple closed paths } \alpha \text{ based at } v \big\} \big\vert=\Big\vert \frac{\{ \text{simple closed paths $\alpha$ based at $v$}\}}{\alpha \sim \alpha^{-1}} \Big\vert$$
simple closed paths \textit{up to inversion}.
\end{definition}

\begin{proposition}\label{prop:GeneOfClosedPaths}
Assume that $v \in E^0$ admits no choices. Then every closed path based at $v$ is a reduced product of simple closed paths based at $v$.
\end{proposition}

\begin{proof}
We claim that any closed path $\alpha$ based at $v$ admits a decomposition
$$\alpha=\alpha' \cdot \gamma^{-1} \beta \gamma,$$
where $\gamma$ is a simple admissible path, $\beta$ is a simple cycle, and $\alpha'$ is a closed path based at $v$ with $\vert \alpha' \vert < \vert \alpha \vert$. An inductive application of this claim surely proves the lemma. \medskip \\
To prove the claim, take $\beta' \le \alpha$ to be minimal with the property that there exists some $\gamma < \beta'$ with $r(\gamma)=r(\beta')$, and write $\beta'=\beta\gamma$. Then $\gamma$ is a simple admissible path by construction, and $\beta$ is a closed path such that the base vertex admits no choices, as seen from Lemma~\ref{lem:ClosedPathNoChoice}. Since the only vertex repetition on $\beta$ happens at the endpoints, it follows from the first part of Proposition~\ref{prop:ClosedPathStandardForm} that $\beta$ is a cycle, hence a simple cycle. It follows immediately from minimality that $\gamma^{-1} \cdot \beta=\gamma^{-1}\beta$, and the concatenation is admissible due to Lemma~\ref{lem:ReducedProduct}. Now write $\alpha=\sigma \beta \gamma$ and define $\alpha' := \sigma \cdot \gamma$. Lemma~\ref{lem:ReducedProduct} guarantees that $\alpha'$ is admissible, hence a closed path, and we clearly have $\alpha=\alpha' \cdot \gamma^{-1}\beta \gamma$. Finally observing that 
$$\vert \alpha' \vert \le \vert \sigma \vert + \vert \gamma \vert < \vert \alpha \vert,$$
the proof is complete.
\end{proof}

\begin{remark}
If $v$ admits no choices, and if $\Lambda$ is a set of representatives for the set of simple closed paths based at $v$ modulo inversion, then $\mathbb{F}_v$ is generated by $\Lambda$ due to Proposition~\ref{prop:GeneOfClosedPaths}. However, $\mathbb{F}_v$ need not be freely generated by $\Lambda$. For instance, both vertices in the graph
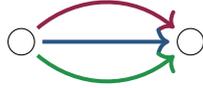
\begin{center}
\begin{tikzpicture}[scale=0.75]
 \SetUpEdge[lw = 1.5pt]
  \tikzset{VertexStyle/.style = {draw,shape = circle,fill = white, inner sep=0pt,minimum size=10pt,outer sep=3pt}}
  \SetVertexNoLabel
  \Vertex[x=0,y=0]{u1}
  \Vertex[x=3,y=0]{u2}

  \tikzset{EdgeStyle/.style = {->,bend left=40, color={\nicered}}}  
  \Edge[](u1)(u2)

  \tikzset{EdgeStyle/.style = {->,color={\niceblue}}}  
  \Edge[](u1)(u2)
  
  \tikzset{EdgeStyle/.style = {->,bend right=40,color={\nicegreen}}}
  \Edge[](u1)(u2)
\end{tikzpicture}
\end{center}
admit three simple closed paths up to inversion, yet $\mathbb{F}_v \cong \mathbb{F}_2$ for either vertex $v$. But this is not a problem since we only need to distinguish between the three cases
\begin{itemize}
\item $\vert \Lambda \vert=0$ in which $\mathbb{F}_v = \{1\}$,
\item $\vert \Lambda \vert = 1$ in which $\mathbb{F}_v \cong \mathbb{Z}$,
\item $\vert \Lambda \vert \ge 2$ in which $\mathbb{F}_v \cong \mathbb{F}_n$ for some $2 \le n \le \infty$.
\end{itemize}
\end{remark}\exend\medskip

We will also need the following somewhat odd corollary.
\begin{corollary}\label{cor:SimpleCycle}
Assume that $v \in E^0$ admits no choices and at most one simple closed path up to inversion. If $v$ admits a cycle $\alpha$, then it admits a unique simple cycle, $\beta$, up to inversion, and $\alpha=\beta^n$ for some $n \in \mathbb{Z}$.  
\end{corollary}
\begin{proof}
If $\alpha$ is a cycle based at $v$, then $\alpha=\beta^n$ for a simple closed path $\beta$ due to Proposition~\ref{prop:GeneOfClosedPaths}. But then $\beta$ must be a cycle as well, hence a simple cycle.
\end{proof}
The proof of Proposition~\ref{prop:GraphCondForOr} is fairly technical and requires the treatment of four different types of edges. We recommend having the following example in mind when reading through the proof of the proposition.

\begin{example}
In the following graph, all edges have been labelled with both the type and a choice of non-separated orientation as defined in the proof of Proposition~\ref{prop:GraphCondForOr}:
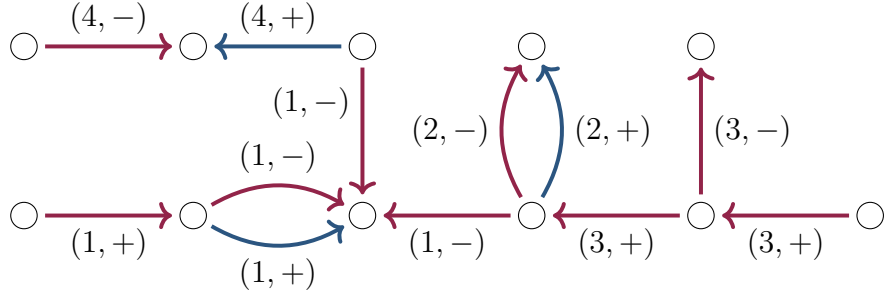
\begin{center}
\begin{tikzpicture}[scale=0.75]
 \SetUpEdge[lw = 1.5pt,labelcolor=none]
  \tikzset{VertexStyle/.style = {draw,shape = circle,fill = white, inner sep=0pt,minimum size=10pt,outer sep=3pt}}
  \SetVertexNoLabel
  \Vertex[x=0,y=0]{u1}
  \Vertex[x=3,y=0]{u2}
  \Vertex[x=6,y=0]{u3}
  \Vertex[x=9,y=0]{u4}
  \Vertex[x=6,y=3]{u5}
  \Vertex[x=3,y=3]{u6}
  \Vertex[x=9,y=3]{u7}
  \Vertex[x=0,y=3]{u10}
  \Vertex[x=-3,y=3]{u11}
  \Vertex[x=12,y=0]{u8}
  \Vertex[x=-3,y=0]{u9}  

  \tikzset{EdgeStyle/.style = {->,color={\nicered}}}  
  \Edge[label={$(1,-)$}, labelstyle=below](u3)(u2)
  \Edge[label={$(3,+)$}, labelstyle=below](u4)(u3)
  \Edge[label={$(3,+)$}, labelstyle=below](u8)(u4)
  \Edge[label={$(1,-)$}, labelstyle=above left](u6)(u2)
  \Edge[label={$(3,-)$}, labelstyle=right](u4)(u7)
  \Edge[label={$(1,+)$}, labelstyle=below](u9)(u1)  
  \tikzset{EdgeStyle/.style = {->,bend left,color={\nicered}}}  
  \Edge[label={$(1,-)$}, labelstyle=above](u1)(u2)
  \Edge[label={$(2,-)$}, labelstyle=left](u3)(u5)  
  \tikzset{EdgeStyle/.style = {->,bend right,color={\niceblue}}}  
  \Edge[label={$(1,+)$}, labelstyle=below](u1)(u2)
  \Edge[label={$(2,+)$}, labelstyle=right](u3)(u5)  
  \tikzset{EdgeStyle/.style = {->,color={\niceblue}}}    
  \Edge[label={$(4,+)$}, labelstyle=above](u6)(u10)
  
  \tikzset{EdgeStyle/.style = {->,color={\nicered}}}
  \Edge[label={$(4,-)$}, labelstyle=above](u11)(u10)
\end{tikzpicture}
\end{center}
The resulting non-separated graph is:
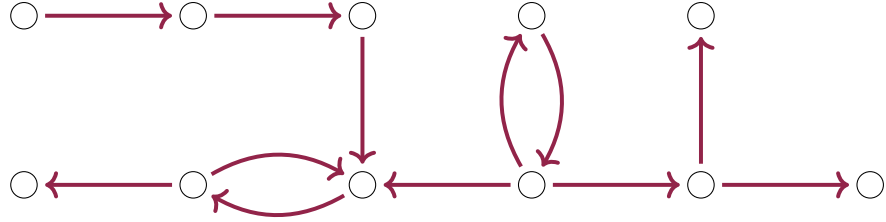
\begin{center}
\begin{tikzpicture}[scale=0.75]
 \SetUpEdge[lw = 1.5pt,labelcolor=none]
  \tikzset{VertexStyle/.style = {draw,shape = circle,fill = white, inner sep=0pt,minimum size=10pt,outer sep=3pt}}
  \SetVertexNoLabel
  \Vertex[x=0,y=0]{u1}
  \Vertex[x=3,y=0]{u2}
  \Vertex[x=6,y=0]{u3}
  \Vertex[x=9,y=0]{u4}
  \Vertex[x=6,y=3]{u5}
  \Vertex[x=3,y=3]{u6}
  \Vertex[x=0,y=3]{u10}
  \Vertex[x=-3,y=3]{u11}
  \Vertex[x=9,y=3]{u7}
  \Vertex[x=12,y=0]{u8}
  \Vertex[x=-3,y=0]{u9}  

  \tikzset{EdgeStyle/.style = {->,color={\nicered}}}  
  \Edge[](u3)(u2)
  \Edge[](u3)(u4)
  \Edge[](u4)(u8)
  \Edge[](u6)(u2)
  \Edge[](u4)(u7)
  \Edge[](u1)(u9)  
  \Edge[](u10)(u6)
  \Edge[](u11)(u10)
  \tikzset{EdgeStyle/.style = {->,bend left,color={\nicered}}}  
  \Edge[](u1)(u2)
  \Edge[](u3)(u5)  
  \tikzset{EdgeStyle/.style = {->,bend left,color={\nicered}}}  
  \Edge[](u2)(u1)
  \Edge[](u5)(u3)  
\end{tikzpicture}
\end{center}\vspace{-0.3cm}
\end{example}\exend\medskip

In the following, we will say that an edge $e$ is \textit{on} a path $\alpha$, if either of the letters $e$ or $e^{-1}$ are present in the symbol expansion of $\alpha$.
\begin{proposition}\label{prop:GraphCondForOr}
Let $(E,C)$ denote a finitely separated Condition \textup{(}C\textup{)} graph, and assume that every vertex admitting no choices admits at most one simple closed path up to inversion. Then $(E,C)$ can be equipped with a non-separated orientation.
\end{proposition}
\begin{proof}
The construction of $\mathfrak{o}(e)$ for $e \in E^1$ will proceed in the following four steps:
\begin{enumerate}
\item Either $e$ or $e^{-1}$ admits a choice (equivalently, $r(e)$ admits a choice).
\item $e$ is not of type (1), but $e$ is on a (simple) cycle.
\item $e$ is not of type (1) or (2), but $e$ is on a (simple) closed path.
\item $e$ is not of type (1), (2), or (3), i.e.~neither $e$ nor $e^{-1}$ admits a choice, and $e$ is not on a closed path.
\end{enumerate}
\textbf{Type (1):} We simply set
$$\mathfrak{o}(e)=\left\{\begin{array}{cl}
1 & \text{if } e \text{ admits a choice} \\
-1 & \text{if } e^{-1} \text{ admits a choice}
\end{array}\right. .$$
Observe that if $v \in E^0$ admits a choice, then exactly one of $e$ and $e^{-1}$ admits a choice for every $e \in r^{-1}(v) \cup s^{-1}(v)$, so before defining $\mathfrak{o}$ on the remaining edges, we might as well check that it satisfies Definition~\ref{def:NonSepOr} at such $v$. If $\mathfrak{o}(e)=1$, i.e.~if $e$ admits a choice, then either some $f \in s^{-1}(r(e))$ with $f \ne e$ or some $X \in C_{r(e)}$ with $e \notin X$ admits a choice. And since every vertex admits at most one choice, we must have $[e]=\{e\}$. Likewise, Definition~\ref{def:NonSepOr}(1) and (2) hold simply because every vertex admits at most one choice.  \medskip \\ 
\textbf{Type (2):} Define an equivalence relation on the set of type (2) edges by
$$e \approx f \Leftrightarrow \text{$e$ and $f$ are on the same (simple) cycle}.$$
Observe that $\approx$ is transitive due to Corollary~\ref{cor:SimpleCycle}. Now choose a representative $e$ for each equivalence class modulo $\approx$ as well as a simple cycle
$$\alpha=e_n^{\varepsilon_n}e_{n-1}^{\varepsilon_{n-1}} \cdots e_2^{\varepsilon_2}e_1^{\varepsilon_1}$$ 
that $e$ is on. We then set $\mathfrak{o}(e_i):=\varepsilon_i$ for all $i=1,\ldots,n$; this is well-defined, because $\alpha$ is simple. \medskip \\
\textbf{Type (3):} If $e$ is a type (3) edge, then $e=\ini(\alpha)$ for a (up to inversion) unique simple closed path $\alpha$. This allows us to define $\mathfrak{o}(e)$ so that $e^{\mathfrak{o}(e)}=\ini_{\textup{d}}(\alpha)$; note that this does not depend on the choice of $\alpha$ over $\alpha^{-1}$ by Proposition~\ref{prop:ClosedPathStandardForm}. \medskip \\
Observe that if $u \in E^0$ admits a closed path but no choices, then we have defined $\mathfrak{o}$ on all edges $e \in r^{-1}(u) \cup s^{-1}(u)$. Before defining the orientation of a type (4) edge, we will therefore check that Definition~\ref{def:NonSepOr} is satisfied at such $u$. We should distinguish between two cases; when $u$ admits and does not admit a cycle. \medskip \\
First assume that $u$ admits a simple cycle $\alpha=e_n^{\varepsilon_n}e_{n-1}^{\varepsilon_{n-1}} \cdots e_2^{\varepsilon_2}e_1^{\varepsilon_1}$, and that the orientation is defined as above. Note that no $e \in r^{-1}(u)$ admits a choice for then $u$ would it self admit a choice, but if $e \in s^{-1}(u)$, then $e^{-1}$ admits a choice if and only if $\vert [e] \vert \ge 2$. Now by construction
$$\mathfrak{o}(e)=\left\{\begin{array}{cl}
-1 & \text{if } e \in s^{-1}(u) \text{ and } \vert [e] \vert \ge 2 \\
\varepsilon_1 & \text{if } e=e_1 \\
\varepsilon_n & \text{if } e=e_n \\
-1 & \text{if } e \in s^{-1}(u) \text{ is not on a cycle} \\
1 & \text{if } e \in r^{-1}(u) \text{ is not on a cycle}
\end{array} \right. .$$
Simply observing that
$$\mathfrak{o}^{-1}(-1) \cap r^{-1}(u) = \left\{\begin{array}{cl}
\emptyset & \text{if } \varepsilon_1=1 \\
\{e_1\} & \text{if } \varepsilon_1=-1
\end{array} \right. $$
and
$$\mathfrak{o}^{-1}(1) \cap s^{-1}(u) = \left\{\begin{array}{cl}
\{e_1\} & \text{if } \varepsilon_1=1 \\
\emptyset & \text{if } \varepsilon_1=-1
\end{array}\right. ,$$
we then see that Definition~\ref{def:NonSepOr} is satisfied at $u$. \medskip \\
Next, assume that $u$ is not on a cycle, but that $\alpha=\gamma^{-1}\beta\gamma$ is a closed path based at $u$. Observe again that all $e \in r^{-1}(u)$ are of type (3), but that $e \in s^{-1}(u)$ is of type (1) if $\vert [e] \vert \ge 2$, and otherwise it is of type (3). If $\ini_{\textup{d}}(\alpha) \in (E^1)^{-1}$, then $e^{-1} \alpha e$ defines a closed path for all $e \in r^{-1}(u)$, $e \ne \ini(\alpha)$, hence $\mathfrak{o}(\ini(\alpha))=-1$ and $\mathfrak{o}(e)=1$. Moreover, if $e \in s^{-1}(u)$, then either $\vert [e] \vert \ge 2$ or $e$ is of type (3) and $e \alpha e^{-1}$ defines a closed path, hence $\mathfrak{o}(e)=-1$ either way. We conclude that Definition~\ref{def:NonSepOr}(1) is satisfied in this case, and similarly one can check that Definition~\ref{def:NonSepOr}(2) is satisfied when $\ini_{\textup{d}}(\alpha) \in E^1$. \medskip \\
\textbf{Type (4):} Finally, let $U \subset E^0$ denote the set of $u \in E^0$ admitting no choices and no closed paths, and define an equivalence relation on $U$ by
$$u \sim v \Leftrightarrow \text{ there is an admissible path of type (4) edges } u \to v.$$
For every equivalence class, we then pick a unique representative. If $e$ is any edge of type (4), and $u$ is the representative of the equivalence class of $r(e)$, then there is a unique admissible path $\alpha$ with $r(\alpha)=u$ and $\ini(\alpha)=e$, and we define $\mathfrak{o}(e)$ so that $e^{\mathfrak{o}(e)} = \ini_{\textup{d}}(\alpha)$. Verifying that $\mathfrak{o}$ satisfies Definition~\ref{def:NonSepOr} is completely analogous to what we did just above.
\end{proof}

\begin{corollary}\label{cor:CondLGraph}
Let $(E,C)$ denote a finitely separated graph satisfying both Condition \textup{(}C\textup{)} and Condition \textup{(}L\textup{)}. Then $(E,C)$ has a non-separated orientation for which the resulting graph $\overline{E}$ satisfies Condition \textup{(}L\textup{)}. 
\end{corollary}

\begin{proof}
Assume that $v \in E^0$ does not admit a choice, and assume in order to reach a contradiction that $v$ admits a closed path $\alpha=\gamma^{-1}\beta\gamma$ with $\beta$ a cycle. Then $s(\beta)$ admits a choice by assumption, hence so does $v$ by Lemma~\ref{lem:ClosedPathNoChoice}. The claim then follows by invoking Proposition~\ref{prop:GraphCondForOr}.
\end{proof}

\section{A characterisation of simplicity}\label{sect:Simplicity}
In this section, we compute all graph algebras of finitely separated graphs giving rise to minimal partial actions, and as a result, we are able to characterise simplicity of these $C^*$-algebras. A similar result is obtained in \cite[Theorem 8.1]{AL} for finite bipartite graphs, but the two proofs are quite different. Indeed, the one in \cite{AL} proceeds via a graph-theoretic investigation of the separated Bratteli diagram $(F_\infty,D^\infty)$, while the below proof combines the contents of Section~\ref{sect:GraphAlgIso} with a simple dynamical observation (Lemma~\ref{lem:CondC}).

\begin{definition}
Let $(E,C)$ denote a finitely separated graph. If $X,Y \in C$ satisfy $\vert X \vert, \vert Y \vert \ge 2$, then a \textit{choice connector between $X$ and $Y$} is an admissible path $\alpha$ for which $Y^{-1}\alpha X$ an admissible composition. If $(E,C)$ does not satisfy Condition (C), we define the \textit{maximal choice distance} to be
$$m_{\textup{CD}}(E,C) := \sup\{ n \mid \text{there exists a choice connector of length } n\}.$$
\end{definition}\exend\medskip

We have the following trivial, but quite handy, observation.

\begin{lemma}\label{lem:inequality}
Given any set $A \subset E^1 \sqcup (E^1)^{-1}$, the function $s_A \colon \Omega(E,C) \to \mathbb{Z}_+$ given by
$$s_A(\xi):= \vert \{ \ini_d(\alpha) \colon \alpha \in \xi, \ter_d(\alpha) \in A\} \vert$$
is lower semi-continuous, i.e.~$\liminf_{\eta \to \xi} s_A(\eta) \ge s_A(\xi)$ for any $\xi \in \Omega(E,C)$.
\end{lemma}
\begin{proof}
This is obvious.
\end{proof}

\begin{lemma}\label{lem:CondC}
Let $(E,C)$ denote a finitely separated graph. If $\theta^{(E,C)}$ is minimal, then $(E,C)$ satisfies Condition \textup{(}C\textup{)}.
\end{lemma}
\begin{proof}
We argue by contraposition, assuming that $(E,C)$ does not satisfy Condition (C). Consider any admissible composition $Y^{-1}\alpha X$ with $\vert X \vert, \vert Y \vert \ge 2$, fix some $x \in X$, $y \in Y$ and set $A:=\{x^{-1},y^{-1}\}$. Also, pick any configuration $\xi$ with $\{x^{-1},y^{-1}\alpha\} \subset \xi$, so that $s_A(\xi) \ge 2$. We may then construct a configuration $\eta$ with the property $\ter_d(\beta) \ne x^{-1},y^{-1}$ for all $\beta \in \eta$. Indeed, starting from any vertex and constructing $\eta$ inductively, one may simply refrain from choosing $x^{-1}$ when reaching $r(x)$, and similarly for $y$ (see Example~\ref{ex:LemmaEx1} for what such a configuration might look like). It follows that $s_A(\theta_\alpha(\eta)) \le 1$ for any $\alpha \in \eta$, so in particular $\xi \notin \overline{\theta_\mathbb{F}(\eta)}$ by Lemma~\ref{lem:inequality}.
\end{proof}

\begin{example}\label{ex:LemmaEx1}
Consider the separated graph
\begin{center}
\begin{tikzpicture}[scale=0.75]
 \SetUpEdge[lw = 1.5pt]
  \tikzset{VertexStyle/.style = {draw,shape = circle,fill = white, inner sep=0pt,minimum size=10pt,outer sep=3pt}}
  \SetVertexNoLabel
  \Vertex[x=-3,y=0]{u1}
  \Vertex[x=0,y=0]{u2}
  \Vertex[x=3,y=0]{u3}

  \tikzset{EdgeStyle/.style = {->,bend left=20,color={\niceblue}}}  
  \Edge[label=$x$, labelstyle=above](u1)(u2)
  \tikzset{EdgeStyle/.style = {->,bend right=20,color={\niceblue}}}    
  \Edge[label=$x'$, labelstyle=below](u1)(u2)  

  \tikzset{EdgeStyle/.style = {->,bend left=20,color={\nicered}}}  
  \Edge[label=$y'$, labelstyle=below](u3)(u2)
  \tikzset{EdgeStyle/.style = {->,bend right=20,color={\nicered}}}    
  \Edge[label=$y$, labelstyle=above](u3)(u2)  
\end{tikzpicture}
\end{center}
which does not satisfy Condition (C). In this case, there is only the following choice of a configuration $\eta$ as in the proof of Lemma~\ref{lem:CondC}:
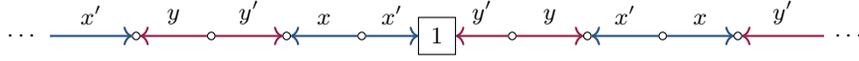
\begin{center}
\begin{figure}[ht]
\begin{tikzpicture}[font=\scriptsize,scale=1]
  \tikzset{VertexStyle/.style = {draw,shape = rectangle,minimum size=14pt,inner sep=1pt}}

  \Vertex[x=0,y=0,L=$1$]{0}     
  
  \tikzset{VertexStyle/.style = {}}  
  \Vertex[x=-5.5,y=0,L=$\cdots$]{-5h}  
  \Vertex[x=5.5,y=0,L=$\cdots$]{5h}     
  
  \tikzset{VertexStyle/.style = {draw,shape = circle,minimum size=1pt,inner sep=1pt}}
  \SetVertexNoLabel  

  \Vertex[x=-1,y=0]{-1}  
  \Vertex[x=-2,y=0]{-2}  
  \Vertex[x=-3,y=0]{-3}    
  \Vertex[x=-4,y=0]{-4}    
  \Vertex[x=1,y=0]{1}  
  \Vertex[x=2,y=0]{2}  
  \Vertex[x=3,y=0]{3}    
  \Vertex[x=4,y=0]{4}  
      
  \tikzset{EdgeStyle/.style = {->,color=\niceblue}}

  \Edge[labelcolor=none, labelstyle=above, label=$x'$](-1)(0)
  \Edge[labelcolor=none, labelstyle=above, label=$x$](-1)(-2)
  \Edge[labelcolor=none, labelstyle=above, label=$x'$](-5h)(-4)

  \Edge[labelcolor=none, labelstyle=above, label=$x'$](3)(2)
  \Edge[labelcolor=none, labelstyle=above, label=$x$](3)(4)
  
  \tikzset{EdgeStyle/.style = {->,color=\nicered}}
  \Edge[labelcolor=none, labelstyle=above, label=$y'$](1)(0)
  \Edge[labelcolor=none, labelstyle=above, label=$y'$](5h)(4)
  \Edge[labelcolor=none, labelstyle=above, label=$y'$](-3)(-2)
  \Edge[labelcolor=none, labelstyle=above, label=$y$](1)(2)
  \Edge[labelcolor=none, labelstyle=above, label=$y$](-3)(-4)  
  ;
\end{tikzpicture}
\caption*{A configuration $\eta$ as in the proof of Lemma~\ref{lem:CondC}.} 
\end{figure}\vspace{-1cm}
\end{center}
\end{example}\exend\medskip

The rest of this section essentially just exploits Condition (C) in order to apply the results of Section~\ref{sect:GraphAlgIso}. However, before we can give the first application of this property, we will need to introduce yet another graph-theoretic notion that will come in handy in the proof of Proposition~\ref{prop:QuoIso}.

\begin{definition}
An admissible path $\alpha$ is called \textit{forced} if $[e]=1$ for all edges $e$, such that $e^{-1}$ (or $e^*$ when regarding $\alpha$ as an element of a graph algebra) is in the symbol expansion of $\alpha$. Observe that $\alpha^* \alpha=s(\alpha)$ whenever $\alpha$ is forced.
\end{definition}

\begin{proposition}\label{prop:QuoIso}
If $(E,C)$ satisfies Condition \textup{(}C\textup{)}, then 
$$L_K(E,C) = L_K^{\textup{ab}}(E,C) \andspace C^*(E,C) = \mathcal{O}(E,C).$$
\end{proposition}

\begin{proof}
We simply have to verify that $\alpha=\alpha\alpha^*\alpha$ in $L_K(E,C)$ for all products of elements from the set $E^1 \cup (E^1)^* \subset L_K(E,C)$. Recall from Remark~\ref{rem:StandardForm} that any such non-zero $\alpha$ is of the form $\alpha=\alpha_n \cdots \alpha_1$, where each $\alpha_i$ is a non-trivial admissible path satisfying
\begin{itemize}
\item $\ini_{\textup{d}}(\alpha_{i+1}) \in E^1$ and $\ter_{\textup{d}}(\alpha_i)=\ini_{\textup{d}}(\alpha_{i+1})^*$,
\item $\vert [\ini_{\textup{d}}(\alpha_{i+1})]\vert \ge 2$
\end{itemize}
for all $i=1,\ldots,n-1$. We first claim that Condition (C) implies $n \le 2$. Indeed if $n \ge 3$, then $\ini_{\textup{d}}(\alpha_2) \in E^1$ and $\ter_{\textup{d}}(\alpha_2) \in (E^1)^*$, so $\vert \alpha_2 \vert \ge 2$ and we can consider the admissible path $\alpha_2'$ obtained from removing the initial and terminal symbol (if $\vert \alpha_2 \vert=2$ so that $\alpha_2=e^*f$ for $e,f \in E^1$, we set $\alpha_2':=r(f)$). Then $\alpha_2'$ will be a choice connector, contradicting Condition (C).  If $n=1$, then $\alpha=\alpha_1$ must be of the form $\alpha=\sigma_2\sigma_1^*$, where both $\sigma_1$ and $\sigma_2$ are forced, and consequently
$$\alpha\alpha^*\alpha=\sigma_2\sigma_1^*\sigma_1\sigma_2^*\sigma_2\sigma_1^*=\sigma_2\sigma_1^*=\alpha.$$
Assuming $n=2$ instead, both $\alpha_1^*$ and $\alpha_2$ must be forced, and so the situation is the same as in the case $n=1$.
\end{proof}

We include a proof of following observation, which is also used in various forms in \cite{AL}, for clarity.

\begin{proposition}\label{prop:NoChoiceIso}
Let $(E,C)$ denote a finitely separated graph, assume that $v \in E^0$ admits no choices. Then there are identifications
$$vL_K^{\textup{ab}}(E,C)v \cong K[\mathbb{F}_v] \andspace v\mathcal{O}^{(r)}(E,C)v \cong C^*_{(r)}(\mathbb{F}_v),$$
where $K[\mathbb{F}_v]$ is the group ring of $\mathbb{F}_v$ with coefficient in $K$.
\end{proposition}
\begin{proof}
Observing that $\Omega(E,C)_v$ is a one-point space and that the partial action of $\mathbb{F}$ restricts to the trivial global action $\mathbb{F}_v \act \Omega(E,C)_v$, we deduce that
$$vL_K^{\textup{ab}}(E,C)v \cong C_K(\Omega(E,C)_v) \rtimes_{\textup{alg}} \mathbb{F}_v \cong K \rtimes_{\textup{alg}} \mathbb{F} \cong K[\mathbb{F}_v]$$
and
$$v\mathcal{O}^{(r)}(E,C)v \cong  C(\Omega(E,C)_v) \rtimes_{(r)} \mathbb{F}_v \cong \mathbb{C} \rtimes_{(r)} \mathbb{F}_v \cong C_{(r)}^*(\mathbb{F}_v),$$
by invoking \cite[Lemma 7.13]{AL}.
\end{proof}

Having made all the preparations, we are now able to describe all algebras associated with finitely separated graphs for which the partial action is minimal.
\begin{theorem}\label{thm:SimplicityMainThm}
Let $(E,C)$ is a finitely separated graph. If $(E,C)$ satisfies Condition \textup{(}C\textup{)} and $\mathcal{H}(E,C)=\{\emptyset, E^0\}$, then exactly one of the following holds:
\begin{enumerate}
\item[\textup{(1)}] Every cycle admits exactly one choice. In that case 
$$L_K(E,C) = L_K^{\textup{ab}}(E,C)$$
is isomorphic to a simple Leavitt path algebra $L(\overline{E})$, and 
$$C^*(E,C) = \mathcal{O}(E,C) \cong \mathcal{O}^r(E,C)$$
is isomorphic to a simple graph $C^*$-algebra $C^*(\overline{E})$.
\item[\textup{(2)}] There is a vertex, which admits no choices and exactly one simple closed path up to inversion. Then $L_K(E,C)=L_K^{\textup{ab}}(E,C)$ is isomorphic to a Leavitt path algebra $L_K(\overline{E})$ and Morita equivalent to the algebra of Laurent polynomials $K[\mathbb{Z}]=K[x,x^{-1}]$, while $C^*(E,C)=\mathcal{O}(E,C) \cong \mathcal{O}^r(E,C)$ is isomorphic to a graph $C^*$-algebra $C^*(\overline{E})$ and Morita equivalent to $C(\mathbb{T})$.
\item[\textup{(3)}] There is a vertex $v \in E^0$, which admits no choices and at least two simple closed paths up to inversion. In that case, there are Morita equivalences
$$L_K(E,C)=L_K^{\textup{ab}}(E,C) \sim K[\mathbb{F}_v], \quad C^*(E,C) = \mathcal{O}(E,C) \sim C^*(\mathbb{F}_v)$$
and $\mathcal{O}^r(E,C) \sim C^*_r(\mathbb{F}_v)$, where $\mathbb{F}_v$ denotes the free subgroup of rank at least two consisting of all the closed paths based at $v$ as well as the empty word.
\end{enumerate}
\end{theorem}
\begin{proof}
First, we recall that the quotient maps $L_K(E,C) \to L_K^{\textup{ab}}(E,C)$ and $C^*(E,C) \to \mathcal{O}(E,C)$ are isomorphisms in any case by Proposition~\ref{prop:QuoIso} (strictly speaking, we only need to invoke the result for case 3). Now, if every cycle admits exactly one choice, then we obtain (1) immediately by Corollary~\ref{cor:CondLGraph}. If this is not the case, then some $v \in E^0$ admits a closed path but no choices, so Proposition~\ref{prop:NoChoiceIso} applies. Moreover, as $v$ generates $E^0$ as a hereditary and $C$-saturated set, it defines a full projection in $L_K(E,C)$, $C^*(E,C)$ and the quotients, hence they are all Morita-equivalent to their respective corners obtained by cutting down with $v$. If $v$ admits at least two simple closed paths up to inversion, then neither can be a multiple of the other, hence $\mathbb{F}_v$ will be a free group of rank at least two. \medskip \\
Finally, we observe that if there is only one simple closed path based at $v$, then no $u \in E^0$ can admit at least two simple closed paths and no choices, for then $C(\mathbb{T})$ and $C^*(\mathbb{F}_n)$ would be Morita-equivalent for some $n \ge 2$. Now Proposition~\ref{prop:GraphCondForOr} applies to give (2).
\end{proof}

As a consequence, we can completely characterise the simple $C^*$-algebras associated with finitely separated graphs.

\begin{corollary}\label{cor:Simple}
Let $(E,C)$ denote a finitely separated graph. Then the algebras $L_K(E,C)$ and $L_K^{\textup{ab}}(E,C)$ as well as the $C^*$-algebras $C^*(E,C)$ and $\mathcal{O}(E,C)$ are simple if and only if the following holds:
\begin{enumerate}
\item[\textup{(1)}] $(E,C)$ satisfies Condition \textup{(}C\textup{)},
\item[\textup{(2)}] $\mathcal{H}(E,C)=\{\emptyset, E^0\}$,
\item[\textup{(3)}] every cycle admits exactly one choice.
\end{enumerate}
In that case, $L_K(E,C) = L_K^{\textup{ab}}(E,C)$ is isomorphic to the Leavitt path algebra and 
$$C^*(E,C) = \mathcal{O}(E,C) \cong \mathcal{O}^r(E,C)$$
is isomorphic to the graph $C^*$-algebra of a non-separated graph.
\end{corollary}
\begin{proof}
If either algebra is simple, then $\theta^{(E,C)}$ is minimal. By Lemma~\ref{prop:NoChoiceIso} and Theorem~\ref{thm:QuoByHS}, this implies that $(E,C)$ satisfies Condition $(C)$ and contains only trivial hereditary and $C$-saturated subsets. Now the result is immediate from Theorem~\ref{thm:SimplicityMainThm}, since (full) group algebras of free groups are not simple.
\end{proof}
\begin{corollary}
Let $(E,C)$ denote a finitely separated graph. Then the $C^*$-algebra $\mathcal{O}^r(E,C)$ is simple if and only if
\begin{enumerate}
\item[\textup{(1)}] $(E,C)$ satisfies Condition \textup{(}C\textup{)},
\item[\textup{(2)}] $\mathcal{H}(E,C)=\{\emptyset, E^0\}$,
\end{enumerate}
and one of the following holds:
\begin{enumerate}
\item[\textup{(3a)}] Every cycle admits exactly one choice. In that case, $\mathcal{O}^r(E,C)$ is isomorphic to a classical graph $C^*$-algebra.
\item[\textup{(3b)}] There is a vertex $v \in E^0$, which admits no choices and at least two simple closed paths up to inversion. In that case, $\mathcal{O}^r(E,C)$ is Morita-equivalent to $C^*_r(\mathbb{F}_v)$, where $\mathbb{F}_v$ denotes the free subgroup of rank at least two consisting of all the closed paths based at $v$ as well as the empty word.
\end{enumerate}
\end{corollary}

\begin{proof}
The proof is completely similar to that of Corollary~\ref{cor:Simple}, except that $C^*_r(\mathbb{F}_n)$ is in fact simple for every $2 \le n \le \infty$ \cite{Powers}.
\end{proof}

Finally, we can also characterise minimality of $\theta^{(E,C)}$:

\begin{corollary}
Let $(E,C)$ denote a finitely separated graph. Then $\theta^{(E,C)}$ is minimal if and only if $(E,C)$ satisfies Condition \textup{(}C\textup{)} and $\mathcal{H}(E,C)=\{\emptyset, E^0\}$.
\end{corollary}
\begin{proof}
One implication follows immediately from \ref{thm:QuoByHS} and Lemma~\ref{lem:CondC}. For the other one, note that Theorem~\ref{thm:SimplicityMainThm} applies, and that if (1) or (3) of the Theorem~\ref{thm:SimplicityMainThm} holds, then $\theta^{(E,C)}$ must be minimal due to simplicity of the graph algebras. Assuming (2) instead, there is a vertex $v$ which admits no choices. Since $v$ generates $E^0$ as a hereditary and $C$-saturated set, we see that $\Omega(E,C)$ is nothing but the orbit of the one-point set $\Omega(E,C)_v$, hence minimal.
\end{proof}

\section{Degeneracy of the tame algebras}

In Section~\ref{sect:GraphAlgIso}, we saw that the Leavitt path algebra and graph $C^*$-algebra degenerate under certain conditions, including Condition (C). On the other hand, even very simple separated graphs without Condition (C) can produce quite complicated algebras. For instance, if $(E,C)$ denotes the graph
\begin{center}
\begin{tikzpicture}[scale=0.75]
 \SetUpEdge[lw = 1.5pt]
  \tikzset{VertexStyle/.style = {draw,shape = circle,fill = white, inner sep=1.5pt,minimum size=10pt,outer sep=3pt}}
  \Vertex[x=4.5,y=3]{v}
  \SetVertexNoLabel
  \Vertex[x=0,y=0]{u1}
  \Vertex[x=3,y=0]{u2}
  \Vertex[x=6,y=0]{u3}
  \Vertex[x=9,y=0]{u4}
  
  \tikzset{EdgeStyle/.style = {->,bend left=30, color={\niceblue}}}  
  \Edge[](u1)(v)

  \tikzset{EdgeStyle/.style = {->,bend left=30,color={\nicered}}}  
  \Edge[](u2)(v)
  
  \tikzset{EdgeStyle/.style = {->,bend right=30, color={\nicered}}}  
  \Edge[](u3)(v)

  \tikzset{EdgeStyle/.style = {->,bend right=30,color={\niceblue}}}  
  \Edge[](u4)(v)
  
\end{tikzpicture}
\end{center}
of \cite[Example 9.4]{AE}, then $C^*(E,C)$ is Morita equivalent to the universal unital $C^*$-algebra generated by two projections , namely
$$vC^*(E,C)v \cong \mathbb{C}^2 \ast_\mathbb{C} \mathbb{C}^2 \cong \{f \in C([0,1],M_2(\mathbb{C})) \mid f(0),f(1) \text{ diagonal}\},$$
while $\mathcal{O}(E,C) \cong \bigoplus_{i=1}^4 M_3(\mathbb{C})$. Indeed, $(E_1,C^1)$ is the trivially separated graph
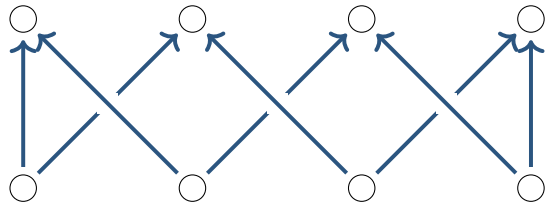
\begin{center}
\begin{tikzpicture}[scale=0.75]
 \SetUpEdge[lw = 1.5pt]
  \tikzset{VertexStyle/.style = {draw,shape = circle,fill = white, inner sep=1.5pt,minimum size=10pt,outer sep=3pt}}
  
  \SetVertexNoLabel
  \Vertex[x=0,y=0]{u1}
  \Vertex[x=3,y=0]{u2}
  \Vertex[x=6,y=0]{u3}
  \Vertex[x=9,y=0]{u4}
  \Vertex[x=0,y=3]{v1}
  \Vertex[x=3,y=3]{v2}
  \Vertex[x=6,y=3]{v3}
  \Vertex[x=9,y=3]{v4}  
  
  \tikzset{EdgeStyle/.style = {->, color={\niceblue}}}  
  \Edge[](u1)(v1)
  \Edge[label=${}$](u1)(v2)
  \Edge[](u2)(v1)
  \Edge[label=${}$](u2)(v3)   
  \Edge[](u3)(v2)
  \Edge[label=${}$](u3)(v4)
  \Edge[](u4)(v3)
  \Edge[](u4)(v4)        
  
\end{tikzpicture}
\end{center}
to which we can apply the standard formula for finite non-separated graphs without cycles. In this short section, we shall explore when the tame algebras degenerate to graph algebras of non-separated graphs by combining our work in Section~\ref{sect:GraphAlgIso} with the fact that $(E_n,C^n)$ and $(E,C)$ produce the same tame algebras. We briefly recall the definition of $(E_1,C^1)$.

\begin{definition}[{\cite[Construction 4.4]{AE}}]
Let $(E,C)$ denote a finite bipartite separated graph, and write
$$C_u=\{X_1^u,\ldots,X_{k_u}^u\}$$
for all $u \in E^{0,0}$. Then $(E_1,C^1)$ is the finite bipartite separated defined by
\begin{itemize}
\item $E_1^{0,0}:=E^{0,1}$ and $E_1^{0,1}:=\{v(x_1,\ldots,x_{k_u}) \mid u \in E^{0,0}, x_j \in X_j^u\}$,
\item $E^1:=\{\alpha^{x_i}(x_1,\ldots,\widehat{x_i},\ldots,x_{k_u}) \mid u \in E^{0,0}, i=1,\ldots,k_u, x_j \in X_j^u \}$,
\item $r_1(\alpha^{x_i}(x_1,\ldots,\widehat{x_i},\ldots,x_{k_u})):=s(x_i)$ and $s_1(\alpha^{x_i}(x_1,\ldots,\widehat{x_i},\ldots,x_{k_u})):=v(x_1,\ldots,x_{k_u})$,
\item $C^1_v:=\{X(x) \mid x \in s^{-1}(v)\}$, where
$$X(x_i):=\{\alpha^{x_i}(x_1,\ldots,\widehat{x_i},\ldots,x_{k_u}) \mid x_j \in X_j^u \text{ for } j \ne i\}.$$
\end{itemize}
We also define a map $\mathfrak{r} \colon E_1^0 \to E^0$ by $\mathfrak{r}(v):=v$ for $v \in E_1^{0,0}=E^{0,1}$ and 
$$\mathfrak{r}(v(x_1,\ldots,x_{k_u})):=u$$
for all $u \in E^{0,0}$ and $(x_1,\ldots,x_{k_u}) \in \prod_{i=1}^{k_u} X_i^u$. \exend
\end{definition}

The following technical lemma will prove most useful.

\begin{lemma}\label{lem:PathCorrespondence}
Assume that $(E,C)$ is a finite bipartite graph. The assignments $v \mapsto \mathfrak{r}(v)$ and $\alpha^e(\ast) \mapsto e^{-1}$ extend to a length-preserving surjective map $\Psi \colon \mathcal{P}(E_1,C^1) \to \mathcal{P}(E,C)$ with the following properties
\begin{enumerate}
\item[\textup{(1)}] If $\alpha, \beta \in \mathcal{P}(E_1,C^1)$ satisfy $r(\alpha)=s(\beta)$, then $\Psi(\beta)\Psi(\alpha)$ is admissible if and only if $\beta\alpha$ is admissible.
\item[\textup{(2)}] If $\alpha \in \mathcal{P}(E,C)$ with $r(\alpha),s(\alpha) \in E^{0,1}$, then $$r_1(\Psi^{-1}(\alpha))=\{r(\alpha)\} \andspace s_1(\Psi^{-1}(\alpha))=\{s(\alpha)\}.$$
\item[\textup{(3)}] If $\alpha \in \mathcal{P}(E,C)$ with $r(\alpha) \in E^{0,0}$ and $s(\alpha) \in E^{0,1}$, so that we may write $\alpha=x\beta$ for $x \in E^1$, then 
$$r_1(\Psi^{-1}(\alpha))=s_1(X(x)) \andspace s_1(\Psi^{-1}(\alpha))=\{s(\alpha)\}.$$
\item[\textup{(4)}] If $\alpha \in \mathcal{P}(E,C)$ is non-trivial with $r(\alpha),s(\alpha) \in E^{0,0}$, so that we may write $\alpha=x \beta y^{-1}$ for $x,y \in E^1$, then
$$
(r_1,s_1)(\Psi^{-1}(\alpha))=s_1(X(x)) \times s_1(X(y)).$$
\item[\textup{(5)}] Let $x\alpha \in \mathcal{P}(E,C)$ and consider a lift $\beta \in \Psi^{-1}(\alpha)$. Then $x\alpha$ is a choice path if and only if $\vert X(x) \vert \ge 2$ and $X(x)^{-1}\beta$ is an admissible composition in $(E_1,C^1)$. Consequently, any $v \in E^{0,1}=E_1^{0,0}$ admits the same number of choices in $(E,C)$ and $(E_1,C^1)$.
\item[\textup{(6)}] The restriction of $\Psi$ to the set of closed paths based at vertices admitting no choices is injective.
\end{enumerate}
\end{lemma}

\begin{proof}
We extend the assignment $\alpha^e(\ast) \mapsto e^{-1}$ to a group homomorphism $\Psi \colon \mathbb{F}(E_1^1) \to \mathbb{F}(E^1)$, and claim that for $e,f \in E_1^1$, the following hold:
\begin{enumerate}
\item[(a)] If $r(e)=r(f)$, then $e^{-1}f$ is admissible if and only if $\Psi(e^{-1})\Psi(f)$ is admissible,
\item[(b)] If $s(e)=s(f)$, then $ef^{-1}$ is admissible if and only if $\Psi(e)\Psi(f^{-1})$ is admissible.
\end{enumerate}
To this end, write $e=\alpha^{x_i}(x_1,\ldots,\hat{x_i},\ldots,x_k)$ and $f=\alpha^{y_j}(y_1,\ldots,\hat{y_j},\ldots,y_l)$. In situation (a), we have
$$s(x_i)=r \big( \alpha^{x_i}(x_1,\ldots,\hat{x_i},\ldots,x_k) \big) = r \big(\alpha^{y_j}(y_1,\ldots,\hat{y_j},\ldots,y_l) \big) = s(y_j),$$
hence $\Psi(e^{-1})\Psi(f)=x_iy_j^{-1}$ is admissible if and only if $x_i \ne y_j$. And since $r(e)=r(f)$, we note that $e^{-1}f$ is admissible if and only if $X(x_i)=[e] \ne [f]=X(y_j)$, which is certainly equivalent to $x_i \ne y_j$. Moving on to (b), we have
$$v(x_1,\ldots,x_k)=s\big(\alpha^{x_i}(x_1,\ldots,\hat{x_i},\ldots,x_k) \big)=s \big( \alpha^{y_j}(y_1,\ldots,\hat{y_j},\ldots,y_l) \big)=v(y_1,\ldots,y_l),$$
so $\Psi(e)\Psi(f^{-1})=x_i^{-1}x_j$ is admissible if and only if $i \ne j$, which is equivalent to $e \ne f$, or $ef^{-1}$ being admissible. It follows that the restriction of $\Psi$ to $\mathcal{P}(E_1,C^1)$ along with the assignment of the vertices defines a length-preserving map $\mathcal{P}(E_1,C^1) \to \mathcal{P}(E,C)$ satisfying (1). Observe, in view of (1), that it is enough to check (2) for admissible paths $\alpha=x_i^{-1}x_j$ of length two, and such $\alpha$ lifts to a path in $ef^{-1}$ with 
$$e:=\alpha^{x_i}(x_1,\ldots,\hat{x_i},\ldots,x_k) \andspace f:=\alpha^{x_j}(x_1,\ldots,\hat{x_j},\ldots,x_k),$$
where $x_l \in X_l^u$ is arbitrary for $l \ne i,j$. (3) and (4) then follow immediate by applying (2) to $\beta$ and invoking (1). In particular, $\Psi$ is surjective. Now consider claim (5) and assume that $X(x)^{-1}\beta$ is an admissible composition with $\vert X(x) \vert \ge 2$. Then $x\alpha$ is in the image of $\Psi$, hence admissible. Moreover, $\vert X(x) \vert \ge 2$ implies that there is some $[x] \ne X \in C_{r(x)}$ with $\vert X \vert \ge 2$, so $x\alpha$ is in fact a choice path. The reverse implication uses the exact same arguments. Finally, consider claim (6) and recall that if $\alpha$ is a closed path and $s(\alpha)$ does not admit any choices, then neither does any vertex on $\alpha$. Consequently, $\Psi$ is injective on the set of edges and vertices that such $\alpha$ may pass through. But then $\Psi$ is surely injective on the set of all such closed paths.
\end{proof}

\begin{corollary}\label{cor:DescMCD}
Let $(E,C)$ denote a finite bipartite separated graph. If $2 \le m_{\textup{CD}}(E,C)< \infty$, then 
$$m_{\textup{CD}}(E_1,C^1)=m_{\textup{CD}}(E,C)-2,$$
and if $m_{\textup{CD}}(E,C)=0$, then $(E_1,C^1)$ satisfies Condition \textup{(}C\textup{)}.
\end{corollary}

\begin{proof}
Simply observe from Lemma~\ref{lem:PathCorrespondence}(5) that if $\beta \in \mathcal{P}(E_1,C^1)$ and $\alpha:=\Psi(\beta)$, then $\beta$ is a choice connector in $(E_1,C^1)$ between $X(x)$ and $X(y)$ if and only if $x\alpha y^{-1}$ is a choice connector in $(E,C)$.
\end{proof}

In the following lemma, finiteness of $E$ is crucial.

\begin{lemma}\label{lem:FinMCD}
Let $(E,C)$ denote a finite separated graph, and assume that every cycle admits at most one choice. Then $m_{\textup{CD}}(E,C) < \infty$.
\end{lemma}
\begin{proof}
Assume in order to reach a contradiction that there is a choice connector $\alpha$ of length $\vert \alpha \vert \ge 3 \cdot \vert E^0 \vert$. Then $\alpha$ must pass some vertex $v \in E^0$ three times, i.e.~there are closed paths $\beta$ and $\gamma$ based at $v$ such that $\gamma\beta$ is admissible. Since $v$ is on a choice connector, it cannot admit any cycles, hence neither $\beta$ nor $\gamma$ are cycles. But then $\gamma\beta$ must itself be a cycle, giving us our desired contradiction.
\end{proof}

\begin{corollary}\label{cor:CondN}
Let $(E,C)$ denote a finite bipartite separated graph. If every cycle admits at most one choice, then $(E_n,C^n)$ will satisfy Condition \textup{(}C\textup{)} for sufficiently large $n$.
\end{corollary}

\begin{proof}
This is immediate from Lemma~\ref{lem:FinMCD} and Corollary~\ref{cor:DescMCD}.
\end{proof}

We now make the final preparations before obtaining the main theorem of this section.

\begin{lemma}\label{lem:AtMostOneClosedPath}
Let $(E,C)$ denote a finite bipartite graph, and assume that every vertex without a choice admits at most one simple closed path up to inversion. Then $(E_1,C^1)$ satisfies the same property.
\end{lemma}

\begin{proof}
Assume that $v \in E_1^0$ does not admit a choice. Without loss of generality, we may assume that $v \in E_1^{0,0}=E^{0,1}$ since every non-trivial path must pass a vertex in this layer. By Lemma~\ref{lem:PathCorrespondence}(5), $v$ does not admit a choice in $(E,C)$ either, hence it admits at most one simple closed path up to inversion in $(E,C)$. It follows from Lemma~\ref{lem:PathCorrespondence}(6) that $v$ admits at most one simple closed path up to inversion in $(E_1,C^1)$ as well.
\end{proof}

\begin{theorem}\label{thm:ODege}
Let $(E,C)$ denote a finite bipartite separated graph. If every cycle admits at most one choice, and every vertex without a choice admits at most simple closed path up to inversion, then $(E_n,C^n)$ admits a non-separated orientation for sufficiently large $n$. Consequently, there exists a finite non-separated graph $F:= \overline{E_n}$ and a direct dynamical equivalence $\theta^F \to \theta^{(E,C)}$. In particular,
$$L_K^{\textup{ab}}(E,C) \cong L_K(F) \andspace \mathcal{O}(E,C) \cong \mathcal{O}^{r}(E,C) \cong C^*(F).$$
\end{theorem}

\begin{proof}
By Corollary~\ref{cor:CondN} and Lemma~\ref{lem:AtMostOneClosedPath}, $(E_n,C^n)$ will satisfy the requirements of Proposition~\ref{prop:GraphCondForOr} for sufficiently large $n$, so the result follows by combining Proposition~\ref{prop:NonSepGraphIso} and \cite[Theorem 3.22]{AL}.
\end{proof}

\section{The exchange property, real rank zero and essentially free actions}

Recall that a non-separated graph $E$ is said to satisfy \textit{Condition} (\textit{K}) if every vertex on a cycle admits at least two simple cycles. The main point of Condition (K) is that it implies Condition (L) and is preserved when passing to any quotient graph $E/H$. It is well known that it is equivalent to $L_K(E)$ being an exchange ring \cite[Theorem 4.5]{APS}, $C^*(E)$ having real rank zero \cite[Theorem 3.5]{Jeong}, and the graph groupoid being essentially principal \cite[Proposition 8]{Paterson}. In this section, we introduce the appropriate generalisation of Condition (K) to finitely separated graphs and prove an analogous result: Condition (K) is equivalent to $L_K^{\textup{ab}}(E,C)$ being an exchange ring, real rank zero of both $\mathcal{O}(E,C)$ and $\mathcal{O}^r(E,C)$, and essential freeness of $\theta^{(E,C)}$. 
\medskip \\
We refer the reader to \cite[Theorem 1.2 and Definition 1.3]{Ara} and \cite[Theorem 2.6]{BP} for various equivalent definitions of exchange rings and real rank zero $C^*$-algebras, respectively. By \cite[Theorem 3.8]{Ara}, these two concepts agree for $C^*$-algebras. The class of exchange rings is closed under ideals, quotients, extensions where idempotents can be lifted modulo the ideal \cite[Theorem 2.3]{Ara}, corners \cite[Corollary 1.5]{ALM}, direct limits, and Morita equivalence between idempotent rings ($C^*$-algebras for instance) \cite[Theorem 2.3]{ALM}.

\begin{definition}\label{def:CondK}
A finitely separated graph $(E,C)$ is said to satisfy \textit{Condition} (\textit{K}) if every vertex $v \in E^0$ on a cycle satisfies the following:
\begin{enumerate}
\item $v$ admits exactly one choice.
\item $v$ admits at least two base-simple cycles up to inversion.
\end{enumerate}
It is apparent that any finite bipartite Condition (K) graph $(E,C)$ satisfies the assumptions of Theorem~\ref{thm:ODege}, so that $\mathcal{O}(E,C)$ will degenerate to a graph $C^*$-algebra $C^*(F)$ with $F:=\overline{E_n}$ for some $n$. However, in order to conclude that $F$ satisfies the usual Condition (K), we first have to check that it is preserved when passing from $(E,C)$ to $(E_n,C^n)$. \exend
\end{definition}
Dealing with base-simple cycles is somewhat complicated in the realm of separated graphs since cycles need not decompose into a product of base-simple cycles. However, when we add Definition~\ref{def:CondK}(1) to the equation, this problem disappears.

\begin{lemma}\label{lem:OneChoiceBaseSimpleDecomp}
Let $(E,C)$ denote a finitely separated graph. If $v \in E^0$ admits exactly one choice, then any cycle based at $v$ is a concatenated product of base-simple cycles
\end{lemma}
\begin{proof}
First observe that whenever $\gamma$ is a cycle based at $v$, exactly one of $\ini_{\textup{d}}(\gamma)$ and $\ter_{\textup{d}}(\gamma)^{-1}$ admits a choice. Now take any cycle $\alpha$ based at $v$ and let $\beta \le \alpha$ denote the minimal closed initial subpath: It suffices to check that $\beta$ must be a cycle. Assume in order to reach a contradiction that it is not, and take a minimal cycle $\beta_n \cdots \beta_1 \le \alpha$ written as a concatenated product of base-simple closed paths with $\beta_1=\beta$. Observe that, by minimality, both $\beta_n\beta_1$ and $\beta_n^{-1}\beta_1$ are cycles. Now if $\ini_{\textup{d}}(\beta)$ admits a choice, then so does $\ini_{\textup{d}}(\beta_n^{-1})=\ter_{\textup{d}}(\beta_n)^{-1}$ and vice versa, contradicting the above observation applied to $\gamma=\beta_n\beta_1$.
\end{proof}

\begin{remark}
It is easy to check that a finitely separated graph $(E,C)$ satisfies Condition (K) if and only if its bipartite sibling $\textbf{B}(E,C)$ satisfies Condition (K). We leave this to the reader.
\end{remark}

\begin{lemma}\label{lem:CondK}
Let $(E,C)$ denote a finite bipartite graph. If $(E,C)$ satisfies Condition \textup{(}K\textup{)}, then so does $(E_1,C^1)$.
\end{lemma}

\begin{proof}
Suppose that $v \in E_1^0$ admits a cycle $\alpha$ in $(E_1,C^1)$; by otherwise replacing $v$ with another vertex on $\alpha$, we may assume that $v \in E_1^{0,0}$. Then $v$ admits the cycle $\Psi(\alpha)$ in $(E,C)$, hence it admits exactly one choice and at least two distinct base-simple cycles in $(E,C)$ by assumption. Using Lemma~\ref{lem:PathCorrespondence}(5), we conclude that it admits exactly one choice in $(E_1,C^1)$ as well, and lifting these cycles arbitrarily to $(E_1,C^1)$ using Lemma~\ref{lem:PathCorrespondence}(2), we obtain two distinct base-simple cycles based at $v$ in $(E_1,C^1)$ as desired.
\end{proof}

\begin{lemma}
Let $(E,C)$ denote a finitely separated graph satisfying Condition \textup{(}K\textup{)}. If $H$ is a hereditary and $C$-saturated set, then the quotient graph $(E/H,C/H)$ satisfies Condition \textup{(}K\textup{)} as well.
\end{lemma}

\begin{proof}
Assume that $v \in (E/H)^0 = E^0 \setminus H$ admits a cycle in $(E/H,C/H)$. Then $v$ admits at least two distinct base-simple cycles in $(E,C)$, and noting that for every cycle $\alpha$, either $\alpha$ or $\alpha^{-1}$ is forced, we see that these cycles are contained in $(E/H,C/H)$ as well. Now if $\beta$ is a minimal choice path with $s(\beta)=v$, then $\beta$ too must be contained in $(E/H,C/H)$, so $v$ admits exactly one choice in $(E/H,C/H)$ as well.
\end{proof}

We will now apply the main result of Section 5.

\begin{corollary}\label{cor:CondKFinBipGr}
Let $(E,C)$ denote a finite bipartite separated graph satisfying Condition \textup{(}K\textup{)}. Then there exists a finite non-separated graph $F$ with Condition \textup{(}K\textup{)} and a direct dynamical equivalence $\theta^F \to \theta^{(E,C)}$. Consequently, $\theta^{(E,C)}$ is essentially free,
$$L_K^{\textup{ab}}(E,C) \cong L_K(F) \andspace \mathcal{O}(E,C) \cong \mathcal{O}^r(E,C) \cong C^*(F).$$
In particular, $L_K^{\textup{ab}}(E,C)$ is an exchange ring and $\mathcal{O}(E,C) \cong \mathcal{O}^r(E,C)$ has real rank zero.
\end{corollary}
\begin{proof}
The first part follows immediately from Theorem~\ref{thm:ODege} with $F=\overline{E_n}$. Moreover, $E_n$ satisfies Condition (K) by Lemma~\ref{lem:CondK}, hence so does $F$ by Lemma~\ref{lem:Paths}. It follows from \cite[Theorem 4.5]{APS} and \cite[Theorem 3.5]{Jeong} that $L_K(F)$ is an exchange ring and $C^*(F)$ has real rank zero, respectively. Moreover, the ideals of $C^*(F)$ are exactly those generated by hereditary and saturated subsets of $F^0$, which correspond to the hereditary and $C^n$-saturated subsets of $E_n^0$ by Proposition~\ref{prop:NonSepGraphIso}. It follows that the closed and invariant subsets of $\Omega(E_n,C^n)$ exactly correspond to the hereditary and $C^n$-saturated subsets of $E_n^0$. Now since $(E_n/H,C^/H)$ satisfies Condition (K) for any such $H$ by Lemma~\ref{lem:CondC}, we see that $\theta^{(E_n,C^n)}$ is essentially free using Theorem~\ref{thm:CondL}. Finally, $\theta^{(E,C)}$ must then be essentially free since there is a direct dynamical equivalence $\theta^{(E_n,C^n)} \to \theta^{(E,C)}$ by \cite[Theorem 3.22]{AL}.
\end{proof}

\begin{corollary}\label{cor:CondKFinGr}
Let $(E,C)$ denote a finite separated graph satisfying Condition \textup{(}K\textup{)}. Then $\theta^{(E,C)}$ is essentially free,
$$L_K^{\textup{ab}}(E,C) \cong L_K(F) \andspace \mathcal{O}(E,C) \cong \mathcal{O}^r(E,C) \cong C^*(F)$$
for a finite graph $F$ with Condition \textup{(}K\textup{)}. In particular, $L_K^{\textup{ab}}(E,C)$ is an exchange ring and $\mathcal{O}(E,C) \cong \mathcal{O}^r(E,C)$ has real rank zero.
\end{corollary}

\begin{proof}
Applying Corollary~\ref{cor:CondKFinBipGr} to $\mathbf{B}(E,C)$, it follows from Corollary~\ref{cor:KEEssFree} and Proposition~\ref{prop:BipDynamics} that $\theta^{(E,C)}$ is essentially free as well. Moreover, there are isomorphisms
$$M_2(L_K^{\textup{ab}}(E,C)) \cong L_K^{\textup{ab}}(\mathbf{B}(E,C)) \cong L_K(F) \andspace  M_2(\mathcal{O}(E,C)) \cong \mathcal{O}(\mathbf{B}(E,C)) \cong C^*(F),$$
where $F$ is a graph satisfying Condition (K), by \cite[Proposition 9.1]{AE}. We may then apply \cite[Theorem 6.1]{AGR} (along with the final comment in the introduction of \cite{AGR}) to obtain a graph $G$ for which $L_K^{\textup{ab}}(E,C) \cong L_K(G)$ and $\mathcal{O}(E,C) \cong C^*(G)$.
\end{proof}

In order to extend Corollary~\ref{cor:CondKFinGr} to arbitrary finitely separated graphs, we need to be able to approximate any such Condition (K) graph by its finite complete Condition (K) subgraphs.

\begin{lemma}\label{lem:CondKLimit}
Every finitely separated Condition \textup{(}K\textup{)} graph is a direct limit of its finite complete Condition \textup{(}K\textup{)} subgraphs.
\end{lemma}

\begin{proof}
Let $(E,C)$ denote a finitely separated Condition (K) graph with a finite complete subgraph $(F,D)$. We then claim that there is an intermediate finite complete subgraph $(F,D) \subset (G,L) \subset (E,C)$ satisfying Condition (K), and we first observe that if $v \in F^0$ admits a cycle $\alpha$ in $(F,D)$, then it automatically admits a choice in $(F,D)$ as well: By assumption, it admits exactly one choice in $(E,C)$, so if $\beta$ is a minimal path with $s(\beta)=v$ leading to a choice $X$, then either $x^{-1}\beta \le \alpha$ or $x^{-1}\beta \le \alpha^{-1}$ for some $x \in X$. From $(F,D)$ being a complete subgraph, it follows that $X \in D$, so $v$ admits exactly one choice in $(F,D)$ as well. \medskip \\
Now assume that $v$ admits only one base-simple cycle in $(F,D)$ up to inversion, and consider some other base-simple cycle $e_n^{\varepsilon_n} \cdots e_1^{\varepsilon_1}$ based at $v$ in $(E,C)$. We then extend the subgraph by the set of edges $\bigcup_{i=1}^n [e_i]$ as well as the ranges and sources of these edges to a finite complete subgraph. Observe that all the added vertices either admit no or at least two base-simple cycles up to inversion, so applying this procedure sufficiently many times leaves us with a finite complete subgraph $(G,L)$ satisfying Condition (K).
\end{proof}

\begin{corollary}\label{cor:CondKInfGr}
Let $(E,C)$ denote a finitely separated graph satisfying Condition \textup{(}K\textup{)}. Then the partial action $\theta^{(E,C)}$ is essentially free, and there are finite non-separated Condition \textup{(}K\textup{)} graphs $(F_n)_{n \ge 1}$ such that 
$$L_K^{\textup{ab}}(E,C) \cong \varinjlim_n L_K(F_n) \andspace \mathcal{O}(E,C) \cong \mathcal{O}^r(E,C) \cong \varinjlim_n C^*(F_n)$$
for appropriate connecting homomorphisms. In particular, $L_K^{\textup{ab}}(E,C)$ is an exchange ring and $\mathcal{O}(E,C) \cong \mathcal{O}^r(E,C)$ has real rank zero.
\end{corollary}

\begin{proof}
By Lemma~\ref{lem:CondKLimit}, we can find an increasing union of finite complete Condition (K) subgraphs $(G_n,L^n)$ of $(E,C)$ such that $(E,C)=\varinjlim_n (G_n,L^n)$. Then $L_K^{\textup{ab}}(G_n,L^n) \cong L_K(F_n)$ and $\mathcal{O}^{(r)}(G_n,L^n) \cong C^*(F_n)$ for some non-separated graph $F_n$ satisfying Condition (K) by Corollary~\ref{cor:CondKFinGr}. Recalling from \cite[Proposition 7.2]{AE2} that $\mathcal{O}$ is a continuous functor and that the same proof applies to $L_K^{\textup{ab}}$, we see that
$$L_K^{\textup{ab}}(E,C) \cong \varinjlim_n L_K(F_n) \andspace \mathcal{O}(E,C) \cong \mathcal{O}^r(E,C) \cong \varinjlim_n C^*(F_n),$$
and as the exchange property passes to limits, it follows from Lemma~\ref{lem:CondKLimit} and Corollary~\ref{cor:CondKFinGr} that $L_K^{\textup{ab}}(E,C)$ is an exchange ring, and $\mathcal{O}(E,C) \cong \mathcal{O}^r(E,C)$ has real rank zero. \medskip \\
We move on to checking essential freeness. Assume that $\Omega \subset \Omega(E,C)$ is a closed invariant subspace and $\xi \in \Omega$ is fixed by $1 \ne \alpha \in \mathbb{F}$. Taking any finite animal $\omega \subset \xi$ with $\alpha \in \omega$, we must verify that $\theta_\alpha^{(E,C)}(\eta) \ne \eta$ for some $\eta \in \Omega \cap \Omega(E,C)_\omega$. By Lemma~\ref{lem:CondKLimit}, there is a finite complete Condition (K) subgraph $(F,D)$ of $(E,C)$ such that $\omega \subset \mathbb{F}(F^1)$, and we consider the canonical surjective $\mathbb{F}(F^1)$-equivariant continuous map 
$$p \colon \Omega(E,C)_{F^0}= \bigsqcup_{v \in F^0} \Omega(E,C)_v \to \Omega(F,D)$$
given by $p(\eta)= \eta \cap \mathbb{F}(F^1)$. Then $\Omega':=p(\Omega(E,C)_{F^0} \cap \Omega)$ is a closed invariant subspace of $\Omega(F,D)$, and $p(\xi) \in \Omega'$ is fixed by $\alpha$. Moreover, $\Omega(F,D)_\omega \cap \Omega'$ is an open neighbourhood of $p(\xi)$ in $\Omega'$, so by essential freeness there is some $\zeta \in \Omega(F,D)_\omega \cap \Omega'$ with $\theta^{(F,D)}_{\alpha}(\zeta) \ne \zeta$. Now any lift $\eta \in \Omega$ of $\zeta$ will do the job.
\end{proof}

Having proved the positive part of the main result of this section, we now begin an investigation of finitely separated graphs not satisfying Condition (K). The lemma just below takes care of the situation in which a cycle admits at least two choices.

\begin{lemma}\label{lem:TwoChoicesConf}
Let $(E,C)$ denote a finitely separated graph. If some cycle admits at least two choices, then there is a configuration $\xi \in \Omega(E,C)$ with stabiliser $\textup{Stab}(\xi) \cong \mathbb{Z}$, such that $\xi$ is isolated in $\overline{\theta_\mathbb{F}(\xi)}$.
\end{lemma}

\begin{proof}
Observe that one of the following holds:
\begin{enumerate}
\item[\textup{(1)}] There is a cycle $\alpha$ and an admissible path $\beta$ with the following properties:
\begin{enumerate}
\item[\textup{(a)}] Both compositions $\beta\alpha$ and $\beta\alpha^{-1}$ are admissible.
\item[\textup{(b)}] $\ter_d(\beta) = x^{-1}$ for some $x \in E^1$ with $\vert [x] \vert \ge 2$. 
\end{enumerate}
\item[\textup{(2)}] There is a cycle $\alpha$ with subpaths $x^{-1} \le \alpha$ and $y^{-1}\beta \le \alpha^{-1}$ such that $\vert [x] \vert, \vert[y] \vert \ge 2$.
\end{enumerate}

In case of (1), we consider the animal $\omega:=\langle \beta\alpha^n \mid n \in \mathbb{Z} \rangle$. Being $\alpha$-periodic, we may extend it to an $\alpha$-periodic configuration $\xi$ such that $\ter_d(\gamma)=x^{-1}$ entails $\gamma \in \omega$ (see also Example~\ref{ex:LemmaEx2} for what such $\xi$ might look like for a particular graph). For the sake of completeness, let us carry out the actual construction of such $\xi$: First consider the finite animal $\langle \ter_d(\alpha)^{-1},\alpha, \beta \rangle$ and extend it arbitrarily to a configuration $\eta$ such that $\ter_d(\gamma)=x^{-1}$ for $\gamma \in \eta$ entails $\gamma \in \langle \ter_d(\alpha)^{-1},\alpha , \beta\rangle$: This can be done as in the proof of Lemma~\ref{lem:CondC} by never choosing to go down $x^{-1}$ when extending. Then consider the animal
$$\chi:=\{\gamma \in \eta \mid \gamma \not\ge \ter_d(\alpha)^{-1},\alpha\}$$
and define $\xi:=\bigsqcup_{n \in \mathbb{Z}} \chi \cdot \alpha^n$. It should be clear that $\xi \in \Omega(E,C)$, and by construction it is fixed by $\alpha$. Let $\gamma \in \xi$ and assume that $\ter_d(\gamma)=x^{-1}$; we may then write $\gamma=\gamma' \cdot \alpha^n$ uniquely with $\gamma' \in \chi$ and $n \in \mathbb{Z}$. Now if $\ter_d(\gamma') = \ter_d(\gamma)=x^{-1}$, we have $\gamma' < \alpha$ or $\gamma' \le \beta$ by construction of $\eta$, and if $\ter_d(\gamma') \ne \ter_d(\gamma)$, then $\gamma'$ must be cancelled out completely by $\alpha^n$, hence $\gamma' < \alpha$. In either case, we see that $\gamma \in \omega$ as required. It follows that $s_{\{x^{-1}\}}(\theta_\gamma(\xi)) \le 2$ whenever $\gamma \in \xi$ is not a power of $\alpha$, while $s_{\{x^{-1}\}}(\xi) = 3$. We conclude from Lemma~\ref{lem:inequality} that $\xi$ is isolated in the closure of its own orbit and that $\text{Stab}(\xi) \cong \mathbb{Z}$. \medskip \\
Now consider (2), and assume without loss of generality that (1) does not hold. In this case, the animal $\omega= \langle \alpha^n \mid n \in \mathbb{Z} \rangle$ can be extended uniquely to a configuration $\xi$, which is necessarily $\alpha$-periodic. Setting $A:=\{x^{-1},y^{-1}\}$, we see that $s_A(\theta_\gamma(\xi)) = 1$ for all $\gamma \notin \omega$ while $s_A(\xi) = 2$, so $\xi$ is once again isolated in $\overline{\theta_\mathbb{F}(\xi)}$ and has stabiliser $\text{Stab}(\xi) \cong \mathbb{Z}$.
\end{proof}

\begin{example}\label{ex:LemmaEx2}
We will now consider the  separated graph
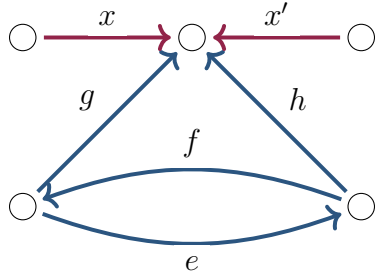
\begin{center}
\begin{tikzpicture}[scale=0.75]
 \SetUpEdge[lw = 1.5pt]
  \tikzset{VertexStyle/.style = {draw,shape = circle,fill = white, inner sep=0pt,minimum size=10pt,outer sep=3pt}}
  \SetVertexNoLabel
  \Vertex[x=-3,y=0]{u1}
  \Vertex[x=0,y=0]{u2}
  \Vertex[x=3,y=0]{u3}
  \Vertex[x=-3,y=-3]{u4}    
  \Vertex[x=3,y=-3]{u5}      

  \tikzset{EdgeStyle/.style = {->,color={\niceblue}}}  
  \Edge[label=$g$, labelstyle=above left](u4)(u2)
  \Edge[label=$h$, labelstyle=above right](u5)(u2)  

  \tikzset{EdgeStyle/.style = {->, bend right=20,color={\niceblue}}}  
  \Edge[label=$e$, labelstyle=below](u4)(u5)  
  \Edge[label=$f$, labelstyle=above](u5)(u4)    

  \tikzset{EdgeStyle/.style = {->, color={\nicered}}}  
  \Edge[label=$x$, labelstyle=above](u1)(u2)
  \Edge[label=$x'$, labelstyle=above](u3)(u2)  

\end{tikzpicture}
\end{center}
 satisfying (1) in the proof of Lemma~\ref{lem:TwoChoicesConf}, and see what the configuration $\xi$ might look like. We take $\alpha$ to be the cycle $\alpha=fe$ and $\beta$ to be the path $\beta=x^{-1}g$, so the animal $\omega$ and the configuration $\xi$ may be pictured as below:
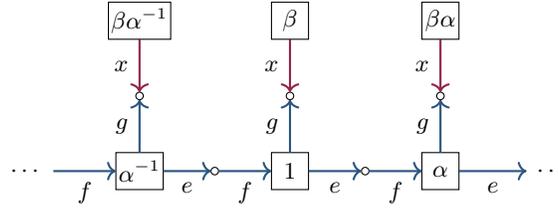
\begin{center}
\begin{figure}[ht]
\begin{subfigure}{.5\textwidth}
  \centering
\begin{tikzpicture}[font=\scriptsize,scale=1]
  \tikzset{VertexStyle/.style = {draw,shape = rectangle,minimum size=14pt,inner sep=1pt}}

  \Vertex[x=0,y=0,L=$1$]{0;0}
  \Vertex[x=2,y=0, L=$\alpha$]{2;0}  
  \Vertex[x=-2,y=0, L=$\alpha^{-1}$]{-2;0} 
  
  \Vertex[x=-2,y=2, L=$\beta\alpha^{-1}$]{-2;2}    
  \Vertex[x=0,y=2, L=$\beta$]{0;2}     
  \Vertex[x=2,y=2, L=$\beta\alpha$]{2;2}     
  
  \tikzset{VertexStyle/.style = {}}  
  \Vertex[x=-3.5,y=0,L=$\cdots$]{-3h;0}  
  \Vertex[x=3.5,y=0,L=$\cdots$]{3h;0}     
  
  \tikzset{VertexStyle/.style = {draw,shape = circle,minimum size=1pt,inner sep=1pt}}
  \SetVertexNoLabel  

  \Vertex[x=-1,y=0]{-1;0}  
  \Vertex[x=1,y=0]{1;0}  
  
  \Vertex[x=-2,y=1]{-2;1}    
  \Vertex[x=0,y=1]{0;1}      
  \Vertex[x=2,y=1]{2;1}
  
  \tikzset{EdgeStyle/.style = {->,color=\niceblue}}

  \Edge[labelcolor=none, labelstyle=below, label=$e$](-2;0)(-1;0)
  \Edge[labelcolor=none, labelstyle=below, label=$e$](0;0)(1;0)
  \Edge[labelcolor=none, labelstyle=below, label=$e$](2;0)(3h;0)

  \Edge[labelcolor=none, labelstyle=below, label=$f$](-3h;0)(-2;0)
  \Edge[labelcolor=none, labelstyle=below, label=$f$](-1;0)(0;0)
  \Edge[labelcolor=none, labelstyle=below, label=$f$](1;0)(2;0)

  \Edge[labelcolor=none, labelstyle=left, label=$g$](-2;0)(-2;1)
  \Edge[labelcolor=none, labelstyle=left, label=$g$](0;0)(0;1)
  \Edge[labelcolor=none, labelstyle=left, label=$g$](2;0)(2;1)    
  
  \tikzset{EdgeStyle/.style = {->,color=\nicered}}
  \Edge[labelcolor=none, labelstyle=left, label=$x$](-2;2)(-2;1)
  \Edge[labelcolor=none, labelstyle=left, label=$x$](0;2)(0;1)
  \Edge[labelcolor=none, labelstyle=left, label=$x$](2;2)(2;1)
  ;
\end{tikzpicture}
  \caption*{An animal $\omega$ as in the proof of Lemma~\ref{lem:TwoChoicesConf}.}
\end{subfigure}
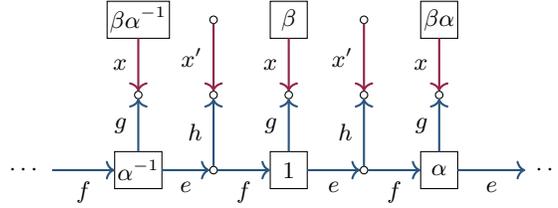
\begin{subfigure}{.5\textwidth}
\begin{tikzpicture}[font=\scriptsize,scale=1]
  \tikzset{VertexStyle/.style = {draw,shape = rectangle,minimum size=14pt,inner sep=1pt}}

  \Vertex[x=0,y=0,L=$1$]{0;0}
  \Vertex[x=2,y=0, L=$\alpha$]{2;0}  
  \Vertex[x=-2,y=0, L=$\alpha^{-1}$]{-2;0} 
  
  \Vertex[x=-2,y=2, L=$\beta\alpha^{-1}$]{-2;2}    
  \Vertex[x=0,y=2, L=$\beta$]{0;2}     
  \Vertex[x=2,y=2, L=$\beta\alpha$]{2;2}     
  
  \tikzset{VertexStyle/.style = {}}  
  \Vertex[x=-3.5,y=0,L=$\cdots$]{-3h;0}  
  \Vertex[x=3.5,y=0,L=$\cdots$]{3h;0}     
  
  \tikzset{VertexStyle/.style = {draw,shape = circle,minimum size=1pt,inner sep=1pt}}
  \SetVertexNoLabel  

  \Vertex[x=-1,y=0]{-1;0}  
  \Vertex[x=1,y=0]{1;0}  
  
  \Vertex[x=-2,y=1]{-2;1}
  \Vertex[x=-1,y=1]{-1;1}      
  \Vertex[x=0,y=1]{0;1}      
  \Vertex[x=1,y=1]{1;1}  
  \Vertex[x=2,y=1]{2;1}
  
  \Vertex[x=-1,y=2]{-1;2}
  \Vertex[x=1,y=2]{1;2}  
    
  \tikzset{EdgeStyle/.style = {->,color=\niceblue}}

  \Edge[labelcolor=none, labelstyle=below, label=$e$](-2;0)(-1;0)
  \Edge[labelcolor=none, labelstyle=below, label=$e$](0;0)(1;0)
  \Edge[labelcolor=none, labelstyle=below, label=$e$](2;0)(3h;0)

  \Edge[labelcolor=none, labelstyle=below, label=$f$](-3h;0)(-2;0)
  \Edge[labelcolor=none, labelstyle=below, label=$f$](-1;0)(0;0)
  \Edge[labelcolor=none, labelstyle=below, label=$f$](1;0)(2;0)

  \Edge[labelcolor=none, labelstyle=left, label=$g$](-2;0)(-2;1)
  \Edge[labelcolor=none, labelstyle=left, label=$g$](0;0)(0;1)
  \Edge[labelcolor=none, labelstyle=left, label=$g$](2;0)(2;1)   
  
  \Edge[labelcolor=none, labelstyle=left, label=$h$](-1;0)(-1;1)
  \Edge[labelcolor=none, labelstyle=left, label=$h$](1;0)(1;1)  
  
  \tikzset{EdgeStyle/.style = {->,color=\nicered}}
  \Edge[labelcolor=none, labelstyle=left, label=$x$](-2;2)(-2;1)
  \Edge[labelcolor=none, labelstyle=left, label=$x$](0;2)(0;1)
  \Edge[labelcolor=none, labelstyle=left, label=$x$](2;2)(2;1)
  
  \Edge[labelcolor=none, labelstyle=left, label=$x'$](-1;2)(-1;1)
  \Edge[labelcolor=none, labelstyle=left, label=$x'$](1;2)(1;1)  
  ;
\end{tikzpicture}
  \caption*{A configuration $\xi$ as in the proof of Lemma~\ref{lem:TwoChoicesConf}.}
\end{subfigure}
\end{figure} 
\vspace{-0.8cm}
\end{center}
\end{example}\exend\medskip

Next, we consider the situation in which a vertex admits exactly one choice and one base-simple cycle up to inversion.
\begin{lemma}\label{lem:OneChoiceOneCycle}
Let $(E,C)$ denote a finitely separated graph, and assume that $v \in E^0$ admits exactly one choice and exactly one base-simple cycle up to inversion. Then there is $H \in \mathcal{H}(E,C)$ with $v \notin H$, such that $v$ admits a cycle but no choices in $(E/H,C/H)$.
\end{lemma}

\begin{proof}
Let $\alpha$ denote the unique base-simple cycle based at $v$ such that $\alpha^{-1}$ is forced. By possibly translating the cycle, we may assume that there is $X \in C_v$ with $\vert X \vert \ge 2$ such that $\ter_{\textup{d}}(\alpha)=x^{-1}$ for some $x \in X$. Now define
$$H:=\{u \in E^0 \mid \text{there is no forced path } v \to u\}$$
and observe that $H$ is hereditary: If $e \in E^1$ and $s(e) \notin H$, i.e.~if there is a forced path $\beta \colon v \to s(e)$, then $e \cdot \beta$ is a forced path $v \to r(e)$ as well, hence $r(e) \notin H$. In order to check that $H$ is also $C$-saturated, we assume that $u \notin H$ with $\beta \colon v \to u$ forced, and take any $Y \in C_u$. If $Y=\{y\}$ is a singleton or $y:=\ter_d(\beta)^{-1} \in Y$, then $y^{-1} \cdot \beta$ is forced as well, so $s(y) \notin H$. If $\vert Y \vert \ge 2$, then we must have $\beta=v$ and $X=Y$ so that $s(x) \notin H$. We conclude that $H$ is indeed a hereditary and $C$-saturated subset. We proceed to check that $v$ does not admit any choices in the quotient graph $(E/H,C/H)$, and it suffices to verify that $X/H=\{x\}$ as the same argument may be applied to any other vertex on $\alpha$. If there were some other $x' \in X/H$, then, by definition of $H$, there would exist a forced path $\beta \colon v \to s(x')$. But then $x'\beta$ would a cycle, which is clearly not a power of $\alpha$, so we have reached a contradiction.
\end{proof}

Finally, we are ready to patch everything together and obtain our main theorem.
\begin{theorem}\label{thm:CondK}
Let $(E,C)$ denote a finitely separated graph. The following are equivalent:
\begin{enumerate}
\item[\textup{(1)}] $(E,C)$ satisfies Condition \textup{(}K\textup{)}.
\item[\textup{(2)}] $\theta^{(E,C)}$ is essentially free.
\item[\textup{(3)}] $\mathcal{O}(E,C)$ has real rank zero.
\item[\textup{(4)}] $\mathcal{O}^r(E,C)$ has real rank zero.
\item[\textup{(5)}] $L_K^{\textup{ab}}(E,C)$ is an exchange ring.
\item[\textup{(6)}] $\mathcal{O}(E,C)$ is the direct limit of real rank zero graph $C^*$-algebras of finite non-separated graphs.
\item[\textup{(7)}] $L_K^{\textup{ab}}(E,C)$ is the direct limit of Leavitt path algebras of finite non-separated graphs with the exchange property.
\end{enumerate}
If $(E,C)$ is finite, then we may replace \textup{(6)} and \textup{(7)} with the conditions
\begin{enumerate}
\item[\textup{(6')}] $\mathcal{O}(E,C)$ is isomorphic to a real rank zero graph $C^*$-algebra of a finite non-separated graph.
\item[\textup{(7')}] $L_K^{\textup{ab}}(E,C)$ is isomorphic to a Leavitt path algebra of a finite non-separated graph with the exchange property.
\end{enumerate}
\end{theorem}

\begin{proof}
In any case, (1) implies (2)-(7) due to Corollary~\ref{cor:CondKInfGr}, and if $(E,C)$ is finite, then (6') and (7') follow from Corollary~\ref{cor:CondKFinGr}. Now suppose that $(E,C)$ does not satisfy Condition (K). Then there is a vertex $v$ on a cycle such that one of the following holds:
\begin{enumerate}
\item[(i)] $v$ admits no choices,
\item[(ii)] $v$ admits exactly one choice and one base-simple cycle up to inversion,
\item[(iii)] $v$ admits at least two choices.
\end{enumerate}
In the case of (i), the compact-open subspace $\Omega(E,C)_v$ is nothing but an isolated point with stabiliser $\mathbb{F}_v$, so the partial action is not even topologically free. Moreover, as we observed in Proposition~\ref{prop:NoChoiceIso},
$$vL_K^{\textup{ab}}(E,C)v \cong K[\mathbb{F}_v], \quad v\mathcal{O}(E,C)v \cong C^*(\mathbb{F}_v) \andspace v\mathcal{O}^r(E,C)v \cong C^*_r(\mathbb{F}_v),$$
so neither is an exchange ring. If (ii) holds, then there is hereditary and $C$-saturated subset $H \subset E^0$ as in Lemma~\ref{lem:OneChoiceOneCycle}, giving rise to an invariant closed subspace on which the restricted action is directly quasi-conjugate to the partial action $\theta^{(E/H,C/H)}$ by Theorem~\ref{thm:QuoByHS}. The quotient graph $(E/H,C/H)$ satisfies (i), so the first case applies. Finally, Lemma~\ref{lem:TwoChoicesConf} applies to (iii) to give a point $\xi$ with non-trivial stabiliser, such that $\xi$ is isolated in $\Omega:=\overline{\theta_\mathbb{F}(\xi)}$. We immediately see that the restricted partial action is not topologically free,
$$1_\xi \big(C_K(\Omega) \rtimes \mathbb{F} \big) 1_\xi \cong K[\mathbb{Z}] \andspace 1_\xi \big( C(\Omega) \rtimes_{(r)} \mathbb{F} \big) 1_\xi \cong C(\mathbb{T}).$$
It follows that none of the above crossed product are exchange rings, so neither are $L_K^{\textup{ab}}(E,C)$, $\mathcal{O}(E,C)$ and $\mathcal{O}^r(E,C)$.
\end{proof}

\begin{corollary}
Let $(E,C)$ denote a finitely separated graph. If either of the algebras $L_K^{\textup{ab}}(E,C)$, $\mathcal{O}(E,C)$ and $\mathcal{O}^r(E,C)$ is an exchange ring, then it is also separative.
\end{corollary}

\begin{proof}
This is immediate from Theorem~\ref{thm:CondK} and \cite[Theorem 3.5, Proposition 4.4 and Theorem 7.1]{AMP}.
\end{proof}

The above corollary shows that the tame algebras of finitely separated graphs do not provide a solution to the Fundamental Separativity Problem for exchange rings. However, as was noted in \cite{AL}, the crossed product $C(\mathcal{X}) \rtimes_\sigma \mathbb{Z}$ of any two-sided subshift is Morita equivalent to a quotient of a separated graph $C^*$-algebra, corresponding to the restriction of the partial action to a closed invariant subspace. In particular, interesting real rank zero $C^*$-algebras, which are not graph $C^*$-algebras, may arise from separated Bratteli diagrams (see \cite[Definition 2.8]{AL}). The question therefore remains if one can find a finite bipartite separated graph $(E,C)$ and a hereditary $D^{\infty}$-saturated subset of the associated separated Bratteli diagram, corresponding to a closed invariant subspace $\Omega \subset \Omega(E,C)$ of infinite type (in the sense of \cite[Section 3]{AL}), such that the monoid 
$$\mathcal{V}(L_K^{\textup{ab}}(E,C)) \cong M(F_\infty/H,D^\infty/H) \cong \varinjlim_n M(E_n/H^{(n)},C^n/H^{(n)}),$$
where $H^{(n)}:=H \cap E_n^0$, is non-separative and the limit algebras 
$$\varinjlim_n L^{\textup{ab}}(E_n/H^{(n)},C^n/H^{(n)}) \cong C_K(\Omega) \rtimes \mathbb{F} \andspace \varinjlim_n \mathcal{O}^r(E_n/H^{(n)},C^n/H^{(n)}) \cong C(\Omega) \rtimes_r \mathbb{F}$$
are exchange rings. One strategy would be to start out with a suitable graph $(E,C)$, for which the monoid $M(E,C)$ is non-separative, and then try to remove more and more obstructions to the exchange property the bigger $n$ gets, while still maintaining injectivity of the composition
$$M(E,C) \to M(E_n,C^n) \to M(E_n,/H^{(n)},C^n/H^{(n)}).$$
One could then hope to remove \textit{all} obstructions to the exchange property in the limit algebras.

\section*{Acknowledgements}
The author would like to thank Pere Ara for many fruitful discussions related to this work, part of which was carried out during a stay at Universitat Aut\`onoma de Barcelona.

\bibliographystyle{plain}
\bibliography{ref}

\begin{thebibliography}{10}

\bibitem{Ara}
P.~Ara.
\newblock Extensions of exchange rings.
\newblock {\em J. Algebra}, 197(2):409--423, 1997.

\bibitem{AE}
P.~Ara and R.~Exel.
\newblock Dynamical systems associated to separated graphs, graph algebras, and
  paradoxical decompositions.
\newblock {\em Adv. Math.}, 252:748--804, 2014.

\bibitem{AE2}
P.~Ara and R.~Exel.
\newblock {$K$}-theory for the tame {$\rm C^*$}-algebra of a separated graph.
\newblock {\em J. Funct. Anal.}, 269(9):2995--3041, 2015.

\bibitem{ALM}
P.~Ara, M.~G\'omez~Lozano, and M.~Siles~Molina.
\newblock Local rings of exchange rings.
\newblock {\em Comm. Algebra}, 26(12):4191--4205, 1998.

\bibitem{AG}
P.~Ara and K.~R. Goodearl.
\newblock {$C^\ast$}-algebras of separated graphs.
\newblock {\em J. Funct. Anal.}, 261(9):2540--2568, 2011.

\bibitem{AG2}
P.~Ara and K.~R. Goodearl.
\newblock Leavitt path algebras of separated graphs.
\newblock {\em J. Reine Angew. Math.}, 669:165--224, 2012.

\bibitem{AGMP}
P.~Ara, K.~R. Goodearl, K.~C. O'Meara, and E.~Pardo.
\newblock Separative cancellation for projective modules over exchange rings.
\newblock {\em Israel J. Math.}, 105:105--137, 1998.

\bibitem{AL}
P.~Ara and M.~Lolk.
\newblock Convex subshifts, separated bratteli diagrams, and ideal structure of
  tame separated graph algebras.
\newblock {\em ArXiv e-prints}, 2017.

\bibitem{AMP}
P.~Ara, M.~A. Moreno, and E.~Pardo.
\newblock Nonstable {$K$}-theory for graph algebras.
\newblock {\em Algebr. Represent. Theory}, 10(2):157--178, 2007.

\bibitem{APS}
G.~Aranda~Pino, E.~Pardo, and M.~Siles~Molina.
\newblock Exchange {L}eavitt path algebras and stable rank.
\newblock {\em J. Algebra}, 305(2):912--936, 2006.

\bibitem{AGR}
S.~E. Arklint, J.~Gabe, and E.~Ruiz.
\newblock Hereditary {$C^\ast$}-subalgebras of graph {$C^\ast$}-algebras.
\newblock arXiv:1604.03085v2 [math.OA].

\bibitem{BCFS}
J.~Brown, L.~O. Clark, C.~Farthing, and A.~Sims.
\newblock Simplicity of algebras associated to \'etale groupoids.
\newblock {\em Semigroup Forum}, 88(2):433--452, 2014.

\bibitem{BP}
L.~G. Brown and G.~K. Pedersen.
\newblock {$C^*$}-algebras of real rank zero.
\newblock {\em J. Funct. Anal.}, 99(1):131--149, 1991.

\bibitem{CL}
T.~M. Carlsen and N.~S. Larsen.
\newblock Partial actions and {KMS} states on relative graph {$C^*$}-algebras.
\newblock {\em J. Funct. Anal.}, 271(8):2090--2132, 2016.

\bibitem{CRS}
T.~M. Carlsen, E.~Ruiz, and A.~Sims.
\newblock Equivalence and stable isomorphism of groupoids, and
  diagonal-preserving stable isomorphisms of graph {$C^*$}-algebras and
  {L}eavitt path algebras.
\newblock {\em Proc. Amer. Math. Soc.}, 145(4):1581--1592, 2017.

\bibitem{CEHS}
L.~O. Clark, C.~Edie-Michell, A.~an~Huef, and A.~Sims.
\newblock Ideals of steinberg algebras of strongly effective groupoids, with
  applications to leavitt path algebras.
\newblock {\em ArXiv e-prints}.

\bibitem{Duncan2}
B.~L. Duncan.
\newblock When universal edge-colored directed graph c*-algebras are exact.
\newblock {\em ArXiv e-prints}.

\bibitem{Exel}
R.~Exel.
\newblock {\em Partial Dynamical Systems, Fell Bundles and Applications}.
\newblock 2014.

\bibitem{ELQ}
R.~Exel, M.~Laca, and J.~Quigg.
\newblock Partial dynamical systems and {$C^*$}-algebras generated by partial
  isometries.
\newblock {\em J. Operator Theory}, 47(1):169--186, 2002.

\bibitem{Goodearl}
K.~R. Goodearl.
\newblock von {N}eumann regular rings and direct sum decomposition problems.
\newblock In {\em Abelian groups and modules ({P}adova, 1994)}, volume 343 of
  {\em Math. Appl.}, pages 249--255. Kluwer Acad. Publ., Dordrecht, 1995.

\bibitem{Jeong}
J.~A. Jeong.
\newblock Real rank of {$C^*$}-algebras associated with graphs.
\newblock {\em J. Aust. Math. Soc.}, 77(1):141--147, 2004.

\bibitem{Li2}
X.~Li.
\newblock Dynamic characterizations of quasi-isometry, and applications to
  cohomology.
\newblock {\em ArXiv e-prints}.

\bibitem{Lolk2}
M.~Lolk.
\newblock On nuclearity and exactness of the tame {$C^*$}-algebras associated
  with finitely separated graphs.
\newblock {\em ArXiv e-prints}, 2017.

\bibitem{Paterson}
A.~L.~T. Paterson.
\newblock Graph inverse semigroups, groupoids and their {$C^\ast$}-algebras.
\newblock {\em J. Operator Theory}, 48(3, suppl.):645--662, 2002.

\bibitem{Powers}
R.~T. Powers.
\newblock Simplicity of the {$C^{\ast} $}-algebra associated with the free
  group on two generators.
\newblock {\em Duke Math. J.}, 42:151--156, 1975.

\end{thebibliography}

\end{document}